\documentclass[12pt]{amsart}
\usepackage{amsmath,amssymb,amsbsy,amstext,amsthm}
\usepackage{accents,color,enumerate,float,verbatim, mathtools}
\usepackage{eucal,bm,kpfonts,mathbbol, cases}

\usepackage[shortlabels]{enumitem}
\usepackage{url, microtype}
\usepackage[colorlinks=true,hyperindex, linkcolor=magenta, pagebackref=false, citecolor=cyan]{hyperref}
\usepackage{tikz,tikz-cd}
\usetikzlibrary{positioning, matrix, shapes}
\usetikzlibrary{decorations.pathmorphing}
\usepackage{lscape}
\usepackage{cleveref}
\usepackage{titlesec}
\titleformat{\section}[block]{\large\scshape\bfseries\filcenter}{\thesection.}{1em}{}		
\titleformat{\subsection}[hang]{\large\scshape\bfseries}{\thesubsection}{1em}{}			
\titleformat{\subsubsection}[hang]{\large\scshape\bfseries}{\thesubsubsection}{1em}{}		

\newtheorem{lemma}{Lemma}[section]
\newtheorem{theorem}[lemma]{Theorem}

\newtheorem{prop}[lemma]{Proposition}
\newtheorem{fact}[lemma]{Fact}
\newtheorem{cor}[lemma]{Corollary}

\newtheorem{question}[lemma]{Question}

\theoremstyle{remark}

\theoremstyle{definition}
\newtheorem{definition}[lemma]{Definition}
\newtheorem{example}[lemma]{Example}
\newtheorem{remark}[lemma]{Remark}

\newtheorem{notation}[lemma]{Notation} 
\newtheorem{construction}[lemma]{Construction}

\newcommand{\Supp}{\operatorname{Supp}}

\newcommand{\Sym}{\operatorname{Sym}} 

\newcommand{\Ass}{\operatorname{Ass}}

\newcommand{\kk}{\mathbf k}
\renewcommand{\aa}{\mathbf a}
\newcommand{\bb}{\mathbf b}

\newcommand{\mm}{\mathbf m}

\newcommand{\xx}{\mathbf x}

\newcommand{\jJ}{\mathbf J}

\newcommand{\cA}{\mathcal{A}}

\newcommand{\cC}{\mathcal{C}}
\newcommand{\cD}{\mathcal{D}}

\newcommand{\cF}{\mathcal{F}}
\newcommand{\cG}{\mathcal{G}}
 
\newcommand{\cI}{\mathcal{I}}

\newcommand{\cL}{\mathcal{L}}

\newcommand{\cP}{\mathcal{P}}
\newcommand{\cQ}{\mathcal{Q}}
\newcommand{\cR}{\mathcal{R}}

\newcommand{\cU}{\mathcal{U}} 		

\newcommand{\NN}{\mathbb{N}}

\newcommand{\RR}{\mathbb{R}}

\newcommand{\ZZ}{\mathbb{Z}}

\newcommand{\ds}{\displaystyle}
\newcommand{\ra}{\rightarrow}

\DeclarePairedDelimiter\abs{\lvert}{\rvert}%

\newcommand{\Ast}{\mathop{\scalebox{2}{\raisebox{-0.2ex}{$\ast$}}}}
\newcommand{\newMatrix}[2]{\left[\Ast_{T\in #1}A_T(#2_T)\right]}
\newcommand{\ideal}[1]{\left\langle #1 \right\rangle}
\newcommand{\lowerIdeal}[1]{\left \langle #1 \right\rangle_{\text{lower}}}

\DeclareMathOperator{\FI}{FI}
\newcommand{\bP}{\mathbf{P}}
\newcommand{\bI}{\mathbf{I}}

\begin{document}

\title[Alexander duals of Sym-invariant Stanley-Reisner Ideals]
{Alexander Duals of Symmetric Simplicial Complexes and Stanley-Reisner Ideals}

\author[A. Almousa]{Ayah Almousa}
\address{University of Minnesota - Twin Cities, 206 Church St SE, Minneapolis, MN-55455}\email{almou007@umn.edu}


\author[K. Bruegge]{Kaitlin Bruegge}
\address{University of Kentucky, 715 Patterson Office Tower, Lexington, KY-40506, USA}
\email{kaitlin\_bruegge@uky.edu}

\author[M. Juhnke-Kubitzke]{Martina Juhnke-Kubitzke}
\address{University of Osnabr\"{uc}k, Albrechtstra\ss e 28a, 49076 Osnabr\"uck, Germany}
\email{juhnke-kubitzke@uni-osnabrueck.de}

\author[U. Nagel]{Uwe Nagel}
\address{University of Kentucky, 715 Patterson Office Tower, Lexington, KY-40506, USA}
\email{uwe.nagel@uky.edu}

\author[A. Pevzner]{Alexandra Pevzner}
\address{University of Minnesota - Twin Cities, 206 Church St SE, Minneapolis, MN-55455}
\email{pevzn002@umn.edu}


\begin{abstract}
Given an ascending chain $(I_n)_{n\in\mathbb{N}}$ of $\Sym$-invariant squarefree monomial ideals, we study the corresponding  chain of Alexander duals $(I_n^\vee)_{n\in\mathbb{N}}$. Using a novel combinatorial tool, which we call \emph{avoidance up to symmetry}, we provide an explicit description of the minimal generating set up to symmetry in terms of the original generators. Combining this result with methods from discrete geometry, this enables us to show that the number of orbit generators of $I_n^\vee$ is given by a polynomial in $n$ for sufficiently large $n$. The same is true for the number of orbit generators of minimal degree, this degree being a linear function in $n$ eventually. The former result implies that the number of $\Sym$-orbits of primary components of $I_n$ grows polynomially in $n$ for large $n$. As another application, we show that, for each $i\geq 0$, the number of $i$-dimensional faces of the associated Stanley-Reisner complexes of $I_n$ is also given by a polynomial in $n$ for large $n$. 
\end{abstract}
\maketitle


\section{Introduction} 

We consider families of simplicial complexes $\Delta_n$ that are invariant under the action of the symmetric group $\Sym (n)$ on $n$ letters acting by permuting the vertices $V_n$ of $\Delta_n$. Such a family is called a  \emph{$\Sym$-invariant chain} $(\Delta_n)_{n \in \NN}$  of 
simplicial complexes if $\Sym( n)$ maps any  face of $\Delta_m$ onto a face of $\Delta_n$ whenever $m \le n$ and the iterated cone over $\Delta_n$ with the new vertices as apices is contained in $\Delta_{n+1}$. If there is only one group of vertices, that is, $\Delta_n$ is supported on $[n] =\{1,\ldots,n\}$, then, for any $n \gg 0$, $\Delta_n$ is the $(n-d)$-dimensional skeleton of an $(n-1)$-dimensional simplex for some fixed integer $d \ge 0$. In other words, for sufficiently large $n$, the facets of $\Delta_n$ are precisely the $\Sym(n)$-orbit of $[n-d+1]$ in this case.   
However, if $\Delta_n$ has $c > 1$ classes of vertices, i.e., $V_n = [n]^c = \{j_i \; \mid  \; i \in [c], j \in [n]\}$, and the action of $\Sym (n)$ maintains each class, that is, $\sigma \cdot j_i = (\sigma(j))_i$ there is considerable variety. Still, it turns out that any $\Sym$-invariant chain $(\Delta)_{n \in \NN}$ \emph{stabilizes} in the sense that there is some integer $n_0$ such that the iterated cone over $\Delta_n$ with apices $V_{n+1} \setminus V_n$ equals $\Delta_{n+1}$ whenever $n \ge n_0$, and so in particular $\dim \Delta_{n+1} = c + \dim \Delta_n$ (see \Cref{prop:stabilization of chain of complexes}). 
Moreover, the total number of facets of $\Delta_n$ grows polynomially in $n$ eventually as a consequence of, e.g., \cite[Theorem 6.5]{N}. 

There is a rich literature relating properties of a simplicial complex $\Delta$ to properties of its Alexander dual $\Delta^{\vee}$ (see, e.g., \cite{BT,ER,HH,HRW,MS,TH}). Motivated by these connections we investigate the family of Alexander duals $(\Delta_n^{\vee})_{n \in \NN}$, given any  $\Sym$-invariant chain 
$(\Delta_n)_{n \in \NN}$  of symmetric simplicial complexes.  By \cite[Theorem 7.10]{NR}, the dimension of $\Delta_n^{\vee}$ will grow linearly in $n$ eventually. However, the growth rate is not always the number of classes of vertices $c$, and so the above stabilization of $(\Delta_n)_{n \in \NN}$ does not extend to stabilization of the sequence of Alexander duals $(\Delta_n^{\vee})_{n \in \NN}$. Moreover, the number of facets of $\Delta_n^{\vee}$ can grow exponentially in $n$. 
Nevertheless, we show that the facets of $\Delta_n^{\vee}$ become eventually predictable. We use this to establish the following result. 

\begin{theorem}
   \label{thm:intro-facet growth}
For any $\Sym$-invariant chain $(\Delta_n)_{n \in \NN}$  of 
simplicial complexes, 
the number of $\Sym(n)$-orbits of facets of $\Delta_n^{\vee}$ grows polynomially in $n$ eventually. 
\end{theorem}

In fact, more is true. Fixing any integer $d \ge 0$, we prove that the number of $\Sym (n)$-orbits of $d$-dimensional faces of $\Delta_n^{\vee}$ grows polynomially in $n$ eventually (see \Cref{prop: fVectorGeneral}).
\smallskip 

By the Stanley-Reisner correspondence, simplicial complexes can be equivalently described by squarefree monomial ideals. Thus, we also consider $\Sym$-invariant chains $(I_n)_{n \in \NN}$ of squarefree monomial ideals $I_n \subset R_n = \mathbf{k} [x_{i,j} \;  \mid  \; i\in [c], j\in [n]]$. Such a chain is a family of ideals $I_n$ with the property that the action of $\Sym (n)$ takes monomials in $I_m$ to monomials in $I_n$ whenever $m \le n$, where the action is defined by $\sigma \cdot x_{i, j} = x_{i, \sigma (j)}$. Here, $\mathbf{k}$ is any field. Such chains were first considered in \cite{HS} and correspond to $\FI$-ideals $\bI$ of a Noetherian polynomial $\FI$-algebra $\bP$ as introduced in \cite{NR2}. Finite generation of $\bI$ is equivalent to stabilization of $(I_n)_{n \in \NN}$ in the sense that there are finitely many monomials 
$\mm_1,\ldots,\mm_t$ such that $I_n$ is minimally generated by the $\Sym(n)$-orbits of $\mm_1,\ldots,\mm_t$ if $n \gg 0$. Such stabilization does not occur in the sequence of Alexander duals $(I_n^{\vee})_{n \in \NN}$. Nevertheless, given the orbit generators $\mm_1,\ldots,\mm_t$ of the original $\Sym$-invariant chain, we describe a minimal generating set of $I_n^{\vee}$ whenever $n \gg 0$. It allows us to show the following result:

\begin{theorem}
   \label{thm:intro-gen-degrees} 
For any $\Sym$-invariant chain $(I_n)_{n \in \NN}$ of squarefree monomial ideals $I_n \subset R_n$, the number of $\Sym(n)$-orbits whose union gives the minimal generators of $I_n^{\vee}$  grows eventually polynomially in $n$. 

Moreover, the least degree of a minimal generator of $I_n^{\vee}$ is given by a linear function $d$ in $n$ if $n \gg 0$. The number of $\Sym(n)$-orbits of monomials whose union is the set of degree $d(n)$ generators of $I_n^{\vee}$  is given by a polynomial in $n$. 
\end{theorem}
Furthermore, we provide a bound of the degree of the mentioned polynomials, which is attained in examples. 

We will show that \Cref{thm:intro-facet growth} is a direct (equivalent) consequence of the first part of this statement. 

Moreover, \Cref{thm:intro-gen-degrees} has consequences for the ideals in the original $\Sym$-invariant chain. By a result of Draisma, Eggermont and Farooq \cite{DEF}, for any $\FI$-ideal $\bI = (I_n)_{n \in \NN}$, the number of $\Sym(n)$-orbits of primary components of $I_n$ grows quasi-polynomially in $n$ eventually. If $\bI$ is a squarefree monomial ideal \Cref{thm:intro-gen-degrees} implies a stronger conclusion. 

\begin{cor}
    \label{cor:intro-number min gens}
If $(I_n)_{n \in \NN}$ is any  $\Sym$-invariant chain of squarefree monomial ideals, then the number of $\Sym(n)$-orbits of primary components of $I_n$ grows polynomially in $n$ eventually. 
\end{cor}
Furthermore, the second part of \Cref{thm:intro-facet growth} implies that the number of $\Sym(n)$-orbits of primary components of $I_n$ with minimum height (the latter also depending on $n$) grows polynomially in $n$ eventually as well. Recall that these components of minimal height determine the height and the degree of $I_n$. 

Our methods are constructive. We explicitly describe the $\Sym(n)$-orbits of monomials that minimally generate the ideal $I_n^{\vee}$, given the monomials that generate the $\Sym$-invariant chain $(I_n)_{n \in \NN}$ of squarefree monomial ideals. Our arguments rely on two combinatorial results that are of independent interest. The first key finding concerns avoidance up to symmetry. In a special case it considers two maps $f, g \colon N \to 2^M$, where $M$ and $N$ are finite sets and $2^M$ is the power set of $M$ partially ordered by inclusion of subsets of $M$. In \Cref{thm: combinatorial statement}, we characterize the existence of a permutation $\sigma \in \Sym (N)$ such that $f(i)$ and $(g \circ \sigma) (i)$ are disjoint for every $i \in N$. In fact, we provide two equivalent conditions for this behavior. Both amount to solubility of certain systems of inequalities among the fibers of $f$ and $g$ and involve upper order ideals of $2^M$. 

Our second key result is in discrete geometry. We consider a polyhedron $P$ in some $\RR^s$ whose supporting hyperplanes are defined by equalities on sums of coordinates. We show that $P$ is either empty or a disjoint union of finitely many pointed rational cones with integral apices (see \Cref{prop: disjoint cone decomposition}). As a consequence, Ehrhart theory implies that certain counts of integer points of $P$ are given by a polynomial. This result will be crucial to determine the number of $\Sym (n)$-orbits of minimal generators of Alexander dual ideals $I_n^{\vee}$. 
\smallskip

We now outline the organization of the article. In \Cref{sec: avoidanceUpToSym} we provide background on order ideals and prove the mentioned result concerning avoidance up to symmetry (see \Cref{thm: combinatorial statement}). \Cref{sec: background} treats invariant chains of ideals and simplicial complexes and relates them via Alexander duality. In \Cref{sec: alexDualsOrderIdeals} we give criteria for containment in the Alexander dual, as well as for divisibility of monomials up to symmetry and being a minimal generator of the Alexander dual (see \Cref{coro: inequalities for alexander dual}, \Cref{lem: divisibility up to symmetry}, \Cref{lem: exchangeLemma}). In \Cref{sec: oneOrbit}, using the results from the previous sections, we provide a description of the minimal generating set (up to symmetry) for the Alexander duals of an invariant chain of squarefree monomial ideals with one orbit generator (\Cref{cor: one orbit minimal gen set} and \Cref{prop: simplified MG_C}). We generalize this result to an arbitrary number of orbit generators in \Cref{sec: generalCase} (\Cref{thm:min gens in general}). The goal of \Cref{sec: coneDecomp} is to prove the mentioned results from discrete geometry that will be essential for counting minimal generators. We use those results together with the description of the minimal generators to prove \Cref{thm:intro-gen-degrees} and \Cref{cor:intro-number min gens} in \Cref{sec: counting}. The last section gives an application to face numbers of chains of simplicial complexes (\Cref{prop: fVectorGeneral}).


\section{Order Ideals and Avoidance up to Symmetry}\label{sec: avoidanceUpToSym}

The goal of this section is to establish a result we dub \emph{avoidance up to symmetry}. 
We begin by recalling definitions and establishing notation for (upper) order ideals in a Boolean poset. As we shall see, order ideals capture information about the columns of exponent matrices of squarefree monomial ideals, and the tools we develop in this section will be central to understanding which orbit classes of squarefree monomials do or do not appear in the Alexander duals of Sym-invariant chains of squarefree monomial ideals.

\subsection{Order Ideals}

\begin{definition}
Fix a finite set $M$ and denote by $2^M$ its Boolean lattice ordered by inclusion.
For any subposet $K$ of $2^M$ and $J\subset K$, say $J$ is an \textit{(upper) order ideal} if $J$ is closed under taking supersets, i.e., whenever $T\in J$ and $S\supset T$, one has $S\in J$.
\end{definition}

\begin{remark}
Unless stated otherwise, an ``order ideal'' is always assumed to be an upper order ideal.
\end{remark} 

For subsets $S$ and $T$ of any set, we denote  its difference set by $S - T$. For an element $t \in T$, we often write for simplicity $S - t$ instead of $S - \{t\}$. If $S$ and $T$ are disjoint we use $S \sqcup T$ to indicate a decomposition of $S \cup T$. 

\begin{notation}
For any $T\subset M$, let $T^C\coloneqq M-T$ denote the complement of $T$ in $M$. Given a subposet $K$ of $2^M$ and any subset $J\subset K$, we denote by $\overline{J}$ the set
$$\overline{J}\coloneqq\{T^C~\mid~ T\in J\}.$$
\end{notation}

We emphasize that $\overline{J}$ consists of the complements of all elements in $J$, it is \textit{not} the complement of $J$ in $2^M$, which we will typically denote by $2^M - J$.

Since the map $\phi: 2^M \ra 2^M$ such that $\phi(T) = T^C$ for any $T\in 2^M$ is a bijection, we have the following useful properties of order ideals.

\begin{lemma}\label{lem: subsets of 2^M}
    Let $I,J\subseteq 2^M$ with $I\subseteq J$. Then the following hold:
    \begin{enumerate}[(a)]
        \item $\overline{I}\subset\overline{J}$;
        \item $\overline{J-I}=\overline{J}-\overline{I}$.
    \end{enumerate}
\end{lemma}

These simple properties imply the following lemma.

\begin{lemma}\label{lem: complement of order ideal}
    Let $K$ be a subposet of $2^M$. For a subset $J\subseteq K$, the following conditions are equivalent: 
    \begin{enumerate}[(a)]
        \item $J$ is an order ideal of $K$.
        \item $\overline{K}-\overline{J}$ is an order ideal of $\overline{K}$. 
    \end{enumerate}
\end{lemma}

\begin{proof}
Assume $(a)$ is true. 
    Fix $S\in\overline{K}-\overline{J}$ and let $T\supset S$. By Lemma \ref{lem: subsets of 2^M}, $\overline{K}-\overline{J}=\overline{K-J}$ and since $S\in\overline{K}-\overline{J}$, it follows that $S^C\notin J$. Since $T^C\subset S^C$ and $J$ is an order ideal, we get $T^C\notin J$, and so $T\in \overline{K}-\overline{J}$, proving $(b)$.
    
    The argument that $(b)$ implies $(a)$ is similar.
\end{proof}

In the following, we use $\NN$ and $\NN_0$ to denote the set of positive and non-negative integers, respectively. 

\subsection{Avoidance up to symmetry}

In this subsection, we prove a combinatorial criterion which detects whether the images of an element $i$ in a set $N$ under two different maps $f$ and $g$ into a subposet of the Boolean lattice share elements ``up to symmetry''; that is, whether $f(i)\cap g(\sigma(i)) = \varnothing$ for some $\sigma\in \Sym(N)$. The next lemma will be crucial to prove this criterion.
\begin{lemma}\label{lem: equivalence of inequalities}
    Let $K$ be a subposet of $2^M$. Consider two maps
    $$K\to\NN_0, \quad T\mapsto k_T,$$
    $$\overline{K}\to \NN_0, \quad S\mapsto \ell_S$$
    satisfying $\sum_{T\in K}k_T=\sum_{T\in K}\ell_{T^C}$. Then the following conditions are equivalent:
    \begin{enumerate}[(a)]
        \item For every order ideal $J$ of $K$, one has $\sum_{T\in J} k_T\leq \sum_{T\in J}\ell_{T^C}$.
        \item For every order ideal $I$ of $\overline{K}$, one has $\sum_{S\in I}\ell_S\leq \sum_{S\in I}k_{S^C}$.
    \end{enumerate}
\end{lemma}

\begin{proof}
    Set $n\coloneqq \sum_{T\in K} k_T=\sum_{T\in K}\ell_{T^C}$. Then for any order ideal $J$ of $K$, one has
    \begin{align*}
        \sum_{T\in J}k_T &= n-\sum_{T\in K-J}k_T,  \\
        \sum_{T\in J}\ell_{T^C} &= n-\sum_{T\in K-J}\ell_{T^C}.
    \end{align*}
    Hence, the inequality $\sum_{T\in J}k_T\leq \sum_{T\in J}\ell_{T^C}$ holds if and only if $\sum_{T\in K-J}\ell_{T^C}\leq \sum_{T\in K-J}k_T$ holds. Since $T\in K-J$ if and only if $T^C\in \overline{K-J}$, the latter inequality is equivalent to
    $$\sum_{S\in \overline{K-J}}\ell_S\leq\sum_{S\in\overline{K-J}}k_{S^C}.$$ 
    By Lemma \ref{lem: complement of order ideal}, $J$ is an order ideal of $K$ if and only if $\overline{K-J}$ is an order ideal of $\overline{K}$, and so  the equivalence of (a) and (b) is proven.
\end{proof}

\begin{theorem}\label{thm: combinatorial statement}
    Let $K$ be a nonempty subposet of $2^M$ and let $N$ be a finite set. Consider maps
    $$f:N\to K$$
    $$g:N\to\overline{K}.$$
     Then the following conditions are equivalent: 
    \begin{enumerate}[(a)]
        \item For every order ideal $J$ of $K$,
        $$\sum_{T\in J}
        \left|f^{-1}(T)\right|\leq 
        \sum_{T\in J}\left|g^{-1}(T^C)\right|.$$
        \item For every order ideal $I$ of $\overline{K}$,
        $$\sum_{S\in I}
        \left|g^{-1}(S)\right|\leq \sum_{S\in I}
        \left|f^{-1}(S^C)\right|.$$
        \item There exists a permutation $\sigma\in\Sym(N)$ such that $f(i)\cap g(\sigma(i))=\varnothing$ for every $i\in N$.
    \end{enumerate}
\end{theorem}    
    \begin{proof}
        For every $T\in K$, set $k_T:=|f^{-1}(T)|$ and for every $S\in \overline{K}$, set $\ell_S:=|g^{-1}(S)|$. Then the assumption reads $\sum_{T\in K}k_T=\sum_{T\in K}\ell_{T^C}$. Lemma \ref{lem: equivalence of inequalities} shows that (a) and (b) are equivalent.
\smallskip
        
Next, we argue that         $(c)$ implies $(a)$.  Fix a permutation $\sigma\in\Sym(N)$ such that $f(i)\cap g(\sigma(i))=\varnothing$ for all $i\in N$. Let $J$ be an order ideal of $K$. Given $i\in N$, it follows from our assumption that $f(i)\subset g(\sigma(i))^C$. Hence, if $f(i)\in J$ then $g(\sigma(i))^C\in J$. Setting
        $$D:=\{i\in N:f(i)\in J\}$$
        $$C:=\{j\in N:g(j)\in\overline{J}\},$$
        we have an injective map
        $$D\hookrightarrow C, \quad i\mapsto\sigma(i).$$
        Hence $|D|\leq |C|$. But $|D|=\sum_{T\in J}k_T$ and $|C|=\sum_{T\in J}\ell_{T^C}$, and so we have proven that Condition $(a)$ is satisfied.
\smallskip
        
Now    we prove that  $(a)$ implies $(c)$ by induction on $|N|$. Note that replacing $g$ by $g\circ \tau$ for any $\tau\in\Sym(N)$ changes neither the assumption nor the assertion.
        
        First assume that $|N|=1$ and $N=\{i\}$. Let $T_0:=f(i)\in K$, and so $k_{T_0}=1$ and all other $k_T=0$. Let $J$ be the order ideal of $K$ that is generated by $T_0$. By our assumption, there exists some $S_0\in J$ such that $\ell_{S_0^C}\geq 1$. Since $S_0\in J$, we have $T_0\subseteq S_0$ and hence $T_0\cap S_0^C=\varnothing$. Moreover, $S_0^C=g(i)$. Taking $\sigma$ to be the identity permutation, we have $$f(i)\cap g(\sigma(i))=T_0\cap S_0^C=\varnothing,$$ as desired.
        
        Now assume that $|N|\geq 2$. If $f(N)=\{\varnothing\}$, then we are done. So assume there is a nonempty $T_0\in K$ such that $k_{T_0}>0$. Then by the same argument as above, there is some $S_0\in K$ such that $T_0\subseteq S_0$ and $\ell_{S_0^C}>0$. Fix an element $i\in f^{-1}(T_0)$. After possibly applying a permutation to $N$, the domain of $g$, one can assume that $f(i)=T_0$ and $g(i)=S_0^C$. Set $\widetilde{N}:= N-\{i\}$, and let $\tilde{f}, \tilde{g}$ be the restrictions of $f,g$ to $\widetilde{N}$, respectively. Let $\tilde{k}_T:=|\tilde{f}^{-1}(T)|$ and $\tilde{\ell}_S:=|\tilde{g}^{-1}(S)|$ for each $T\in K$ and each $S\in\overline{K}$. We now consider two cases.
        
        \textit{Case 1:} For every order ideal $J$ of $K$, one has
        $$\sum_{T\in J}\tilde{k}_T\leq \sum_{T\in J}\tilde{\ell}_{T^C}.$$
        Then, by induction, there is a permutation $\tilde{\sigma}\in\Sym(\widetilde{N})$ such that $\tilde{f}(j)\cap\tilde{g}(\sigma(j))=\varnothing$ for all $j\in\widetilde{N}$. Let $\sigma\in\Sym(N)$ be the permutation defined by
        $$\sigma(j)=\begin{cases} \tilde{\sigma}(j) & \text{if } j\in\widetilde{N} \\
        i & \text{if }j=i\end{cases}.$$
        Since $\tilde{\sigma}$ is a bijection from $\widetilde{N}$ to $\widetilde{N}$, the map $\sigma$ is a well-defined permutation of $N$. Furthermore, one has that $f(j)\cap g(\sigma(j))=T_0\cap S_0^C=\varnothing$ for all $j\in N$, which proves Condition $(c)$.
        
        \textit{Case 2:} There exists an order ideal $J_0$ of $K$ such that
        \begin{equation}\label{eq: tilde case 2}
        \sum_{T\in J_0}\tilde{k}_T > \sum_{T\in J_0}\tilde{\ell}_{T^C}.
        \end{equation}
        Note that $\tilde{k}_{T_0}=k_{T_0}-1$ and $\tilde{k}_T=k_T$ for all $T\neq T_0$. Similarly, $\tilde{\ell}_{S_0^C}=\ell_{S_0^C}-1$ and $\tilde{\ell}_S=\ell_S$ for all  $S\neq S_0^C$. Since $\sum_{T\in J_0} k_T \le \sum_{T\in J_0}\ell_{T^C}$ by Condition $(a)$, in order for (\ref{eq: tilde case 2}) to occur, we must have had
        \begin{equation}\label{eq: equal sum J_0}
        \sum_{T\in J_0} k_T=\sum_{T\in J_0}\ell_{T^C}
        \end{equation}
        and $T_0\in K-J_0$, $S_0\in J_0$.
        Define subsets $N',N''$ of $N$ as follows:
        \begin{align*}
            N'&:=\{j\in N \; \mid \; f(j)\in J_0\}, \\
            N'' &:= \{j\in N \; \mid \; f(j)\in K-J_0\}. 
        \end{align*}
        Then we get the set partition $N=N'\sqcup N''$. Note that $i\in N''$, since $f(i)=T_0$ and $T_0\in K-J_0$. Hence $N''$ is nonempty. Moreover, combining Equation \eqref{eq: equal sum J_0} with the assumption that $\sum_{T\in K}k_T=\sum_{T\in K}\ell_{T^C}$, we conclude that
        $$|N''|=\sum_{T\in K-J_0}k_T=\sum_{T\in K-J_0}\ell_{T^C}<|N|.$$
        The strict inequality holds since $S_0\in J_0$ and $\ell_{S_0^C}\geq 1$. Hence, both $N'$ and $N''$ are proper, nonempty subsets of $N$. Furthermore, the map $f$ induces maps $f' \colon N'\to J_0$ and $f'' \colon N''\to K-J_0$. 

Now use the map $g$ to define subsets $M', M''$ of $N$ as follows:
        \begin{align*}
            M'&:=\{j\in N \; \mid \;  g(j)\in \overline{J_0}\}, \\
            M'' &:= \{j\in N \; \mid \; g(j)\in \overline{K}- \overline{J_0}\}. 
        \end{align*}
Thus, $M' \sqcup M''$ is a partition of $N$ and $|M' | = \sum_{T \in J_0} \ell_{T^C}$.  Combined with Equation \eqref{eq: equal sum J_0}, we get $|M'| = |N'|$, and so $|M''| = |N''|$. Hence, there is a bijection $\sigma \colon N \to N$ satisfying $\sigma (M') = N'$ and $\sigma (M'') = N''$. It follows that after applying $\sigma^{-1}$ to the domain of $g$, we obtain a map, call it again $g$, with    $g(N')\subseteq \overline{J_0}$ and $g(N'')\subseteq \overline{K}-\overline{J_0}$.  Thus,  $g$ induces maps $g' \colon N'\to\overline{J_0}$ and $g'' \colon N''\to\overline{K}-\overline{J_0}$. Our goal is to apply the induction hypothesis on the pairs of maps $f',g'$ and $f'',g''$.
        
        First consider the maps $f':N'\to J_0$ and $g':N'\to \overline{J_0}$. Since $f^{-1}(T)=(f')^{-1}(T)$ whenever $T\in J_0$ and $g^{-1}(T^C)=(g')^{-1}(T^C)$ whenever $T\in J_0$, Equation \eqref{eq: equal sum J_0} implies
        $$\sum_{T\in J_0}\left|(f')^{-1}(T)\right|=\sum_{T\in J_0} \left|(g')^{-1}(T^C)\right|.$$
        Moreover, since any order ideal $J$ of $J_0$ is also an order ideal of $K$, we obtain 
        $$\sum_{T\in J}\left|(f')^{-1}(T)\right| = \sum_{T\in J}|f^{-1}(T)|\leq \sum_{T\in J}|g^{-1}(T^C)|=\sum_{T\in J}\left|(g')^{-1}(T^C)\right|.$$
        Hence, the induction hypothesis applies to $f', g'$, and so there exists a permutation $\sigma'\in\Sym(N')$ such that for each $j\in N'$, we have $f'(j)\cap g'(\sigma'(j))=\varnothing$.
        
        Now consider $f'' \colon N''\to K-J_0$ and $g'' \colon N''\to\overline{K}-\overline{J_0}$. By Lemmas  \ref{lem: subsets of 2^M} and \ref{lem: complement of order ideal}, we know that $\overline{K}-\overline{J_0}=\overline{K-J_0}$ is an order ideal of $\overline{K}$. It follows from our assumptions on $f$ and $g$ and from Equation \eqref{eq: equal sum J_0}, that
        $$\sum_{T\in K-J_0}\left|(f'')^{-1}(T)\right|=\sum_{T\in K-J_0}\left|(g'')^{-1}(T^C)\right|.$$
        Given an order ideal $I$ of $\overline{K}-\overline{J_0}$, it is also an order ideal of $\overline{K}$. Hence by equivalence of (a) and (b) we have
        $$\sum_{S\in I}\left|(g'')^{-1}(S)\right| = \sum_{S\in I} |g^{-1}(S)| \leq \sum_{S\in I} |f^{-1}(S^C)| = \sum_{S\in I} \left| (f'')^{-1}(S^C)\right|.$$
        Thus, the induction hypothesis applies to $f'', g''$, and so there exists a permutation $\sigma''\in \Sym(N'')$ such that $f''(j)\cap g''(\sigma''(j))=\varnothing$ for every $j\in N''$.
        
        Finally, define a permutation $\sigma\in \Sym(N)$ by 
        $$\sigma(j)=\begin{cases}\sigma'(j) & \text{if } j\in N' \\
        \sigma''(j) & \text{if }j\in N''\end{cases}.$$
        Then $\sigma$ is the desired permutation which satisfies Condition (c) of the theorem.
    \end{proof}

\section{$\Sym$-invariant Chains and Alexander Duality}\label{sec: background}

We now recall some concepts and results that will be used throughout this paper. For further information on combinatorial commutative algebra we refer to \cite{MS}. 

\subsection{Invariant Chains of Ideals}
\begin{notation}\label{not: main} We denote by $\NN$ the set of positive integers.  Fix $c\in \NN$ and let $\kk$ be any field. Set $R = \kk[x_{i,j} \mid i\in [c],j\geq 1]$ to be a polynomial ring in infinitely many variables. One may view $R$ as the colimit of the ascending chain of polynomial rings
$$
R_1\subseteq R_2\subseteq \dots \subseteq R_n\subseteq \dots
$$
where  $R_n = \mathbf{k}[x_{i,j} \mid i\in [c],  j\in [n]]$ for $n \in \NN$. 

Set $\Sym(\infty) = \bigcup_{n\geq 1} \Sym(n)$ to be the infinite symmetric group, where $\Sym(n)$ is the symmetric group on $\{1,\dots,n\}$. Embedding $\Sym (n)$ as the stabilizer of $n+1$ in $\Sym(n+1)$, one may view $\Sym(\infty)$ as the colimit of the ascending chain of finite groups
$$
\Sym(1) \subseteq \Sym(2) \subseteq \dots \subseteq \Sym(n)\subseteq \dots .
$$
\end{notation}

We always consider the action of $\Sym(\infty)$ on $R$ that is induced by $\sigma\cdot x_{i,j} = x_{i,\sigma(j)}$ for $\sigma\in\Sym(\infty)$, and this action restricts to an action of $\Sym(n)$ on the polynomial ring $R_n$.

\begin{definition}\label{def: symInvariant} Adopt Notation \ref{not: main}. An ideal $I\subseteq R$ is said to be \emph{$\Sym(\infty)$-invariant} (or \emph{$\Sym$-invariant} for short) if $\sigma(f)\in I$ for any $f\in I$ and any $\sigma\in\Sym(\infty)$. A \emph{$\Sym(\infty)$-invariant chain} is a sequence $(I_n)_{n\in\NN}$ of ideals $I_n\subseteq R_n$ satisfying
$$
\Sym(n)(I_m) = \left\{\sigma(f)\mid f\in I_m, \ \sigma\in \Sym(n)\right\} \subseteq I_n \text{ whenever } m\leq n.
$$
\end{definition}

The colimit $I = \lim I_n$ of any such chain is isomorphic to the union of the extension ideals $I_n$ in $R$. Thus, $I$ is a $\Sym (\infty)$-invariant ideal of $R$. 
By results of Cohen, Aschenbrenner-Hillar, and Hillar-Sullivant, \cite{C, AH, HS} the ring $R$ is \emph{$\Sym(\infty)$-Noetherian}, meaning that any $\Sym(\infty)$-invariant ideal is generated by finitely many $\Sym(\infty)$-orbits. This gives one the ability to study $\Sym$-invariant ideals via orbit classes of generators.

\subsection{Alexander duality and Stanley-Reisner complexes}
We will restrict our attention to the case of $\Sym$-invariant squarefree monomial ideals, also known as Stanley-Reisner ideals. This is a particularly interesting class of ideals because homological information about these chains of ideals also yields topological information about the associated chains of simplicial complexes with $\Sym$-invariant facets. We begin by recalling some essential tools in the theory of squarefree monomial ideals, such as Alexander duality and Stanley-Reisner theory.

\begin{notation}
Let $S = \mathbf{k}[x_1,\dots, x_n]$ be a polynomial ring over a field $\mathbf{k}$ in $n$ variables. The monomial $\mathbf{X} = x_1 x_2 \cdots x_n$ denotes the product of all the variables.
Given a simplicial complex $\Delta$ on $n$ vertices, we will frequently refer to faces of simplicial complexes as monomials $\mm\in S$ rather than specifying the subset of $[n]$ corresponding to the face.
If $\mm$ is a monomial in $S$, write $P_\mm$ for the prime ideal generated by the variables dividing $\mm$, that is, $P_\mm \coloneqq \langle x_i \mid x_i \text{ divides } \mm\rangle$.
\end{notation}

\begin{definition}
Let $I$ be a squarefree monomial ideal in a polynomial ring $S$. Then the \textit{Stanley-Reisner complex} of $I$, denoted $\Delta(I)$, is the simplicial complex consisting of squarefree monomials \textit{not} in $I$,
$$
\Delta_I = \{\mm \text{ a squarefree monomial in } S \mid \mm\notin I\}.
$$
For a simplicial complex $\Delta$, the \textit{Stanley-Reisner ideal} of $\Delta$ is the squarefree monomial ideal generated by the non-faces of $\Delta$,
$$
I_\Delta = \langle \mm \text{ a monomial in } S \mid \mm\notin\Delta\rangle.
$$
\end{definition}

It is well-known that the above constructions give bijections between the set of squarefree monomial ideals of $S$ and the set of  simplicial complexes on the vertex set $[n]$.

\begin{definition}
If $\Delta$ is a simplicial complex on $[n]$, the \textit{Alexander dual} of $\Delta$, denoted by $\Delta^\vee$, is the simplicial complex with faces $\{ \frac{\mathbf{X}}{\mm} \mid \mm\notin \Delta\}$. In particular, the facets of $\Delta^\vee$ are exactly the complements of the minimal non-faces of $\Delta$.

If $I$ is a squarefree monomial ideal, then the \textit{Alexander dual} of $I$ is
$$
I^\vee \coloneqq \langle \mm \mid P_\mm \in \Ass(I)\rangle,
$$
where $\Ass(I)$ denotes the set of associated primes of $I$. In this way, associated primes of $I$ correspond to the generators of $I^\vee$.
\end{definition}

The two concepts of Alexander duals are compatible because $I_{\Delta}^{\vee}$ is the Stanley-Reisner ideal of $\Delta^{\vee}$. 

The following fact follows from tracing through the definitions of Alexander duals. It is a useful criterion for translating between generators of the Alexander dual of an ideal and the facets of its associated Stanley-Reisner complex.

\begin{fact}\label{fact: alexDualComplements}
Let $\Delta$ be a simplicial complex. Then the facets of $\Delta^\vee$ are the monomials $\frac{\mathbf{X}}{\mm}$ where $\mm$ ranges over the monomial minimal generators of $I_\Delta$. Analogously, if $I$ is a squarefree monomial ideal, then the generators of $I^\vee$ are the monomials $\frac{\mathbf{X}}{F}$, where $F$ ranges over the facets of $\Delta(I)$.
\end{fact} 

Note that any monomial ideal $I$ has a unique minimal generating set $G(I)$ consisting only of monomials. For brevity, we often refer to an element of $G(I)$ simply as a minimal generator of $I$. 

Using the minimal generating set of a squarefree monomial ideal, one can characterize monomials in its Alexander dual more directly. This is well-known and follows, for example, from \cite[Corollary 5.25]{MS}. 

\begin{fact}
      \label{fact:mon in dual}
Consider an ideal $I \subset S$ that is minimally generated by squarefree monomials $\mm_1,\ldots,\mm_t$. Then a squarefree monomial $\mm \in S$ is not in $I^{\vee}$ if and only if there exists $1\leq i\leq t$ such that the monomial $\mm \cdot \mm_i$ is squarefree. 
\end{fact}

\subsection{Alexander Duals of Sym-Invariant Chains}\label{sect:AlexDualChain}

We consider the following question. 

\begin{question}
Given a $\Sym$-invariant chain $(I_n)_{n\in\NN}$ of ideals with $I_n\subseteq R_n$ generated by $\Sym$-orbits of squarefree monomials, what are the generators of its Alexander dual $I_n^\vee$ (or, by taking complements, the facets of their associated Stanley-Reisner complexes $\Delta(I_n)$)?
\end{question}

We will observe that the Alexander dual of $I_n$ is also a $\Sym$-invariant ideal, implying that we can study this question by aiming to understand monomial orbit generators of $I_n^\vee$. This allows us to refine the above question. 

We call a simplicial complex  $\Delta$ on vertex set $[n]$ \emph{$\Sym$-invariant} if $F\in \Delta$ implies $\sigma(F)\in \Delta$ for all $\sigma \in \Sym(n)$. In particular, $F$ and $\sigma(F)$ have the same dimension in this case. The next proposition relates $\Sym$-invariant complexes and ideals.

\begin{prop}\label{prop: symmetricFaces} 
Let $\Delta$ be a simplicial complex on vertex set $[n]^c$. Then $\Delta$ is $\Sym(n)$-invariant if and only if $I_{\Delta}$ is $\Sym(n)$-invariant.
\end{prop}                           

\begin{proof}
First suppose that $I_\Delta$ is $\Sym(n)$-invariant. Let $F\in \Delta$ and assume by contradiction that $\sigma(F)\notin \Delta$, i.e., $\sigma(F)\in I_\Delta$. Since $I_\Delta$ is $\Sym$-invariant, it follows that $F\in I_\Delta$ and hence $F\notin \Delta$ which yields a contradiction. The other direction follows by a similar argument.
\end{proof}

The next proposition shows that $\Sym$-invariance is preserved under taking Alexander duals.

\begin{prop}\label{prop:alexDualComplex}
A simplicial complex $\Delta$ is $\Sym(n)$-invariant if and only if $\Delta^\vee$ is $\Sym(n)$-invariant.
\end{prop}

\begin{proof} 
Let $\Delta$ be $\Sym(n)$-invariant and let $V$ be the vertex set of $\Delta$. Assume by contradiction that there exist $F\in \Delta^\vee$ and $\sigma\in \Sym(n)$ with $\sigma(F)\notin \Delta^\vee$. Interpreting faces of $\Delta$ as monomials, it follows that $\frac{\textbf{X}}{\sigma(F)}\in \Delta$. Since $\Delta$ is $\Sym(n)$-invariant, we have that $\frac{\textbf{X}}{F}\in \Delta$, i.e., $F\notin \Delta^\vee$, which is a contradiction. The other direction follows from the first one since $(\Delta^\vee)^\vee=\Delta$.
\end{proof}

Combining \Cref{prop: symmetricFaces} and \Cref{prop:alexDualComplex} gives the following result.
\begin{prop}\label{prop: alexDual} A squarefree monomial ideal $I$ of $R_n$ is $\Sym(n)$-invariant if and only if its Alexander dual $I^\vee$ is $\Sym(n)$-invariant.
\end{prop}
%

The following useful fact is likely known to experts. 

\begin{prop}\label{prop: reverse chain}
If $I,J$ are squarefree monomial ideals of $S$ with $I\subseteq J$, then $J^\vee\subseteq I^\vee$.
\end{prop}

\begin{proof} For the convenience of the reader, we provide an argument. 

Denote by $\Delta (I)$ and $\Delta (J)$ the simplicial complexes associated to $I$ and $J$. Let $\mm_J$ be a minimal generator of $J^\vee$. Thus, $\mm_J$ corresponds to the complement of a facet $F_J$ of $\Delta_J$. Since $I\subseteq J$, we have that $\Delta_J$ is a subcomplex of $\Delta_I$. Hence $F_J$ is contained in some facet $F_I$ of $\Delta_I$, and so the complement $\mm_I \in I^{\vee}$ of $F_I$ divides $\mm_J$, which shows $\mm_J \in I^{\vee}$. 
\end{proof}

The next corollary is immediate from \Cref{prop: reverse chain}.
\begin{cor}\label{cor: AD reverse chain}
If  $(I_n)_{n\in\NN}$  is a $\Sym$-invariant  chain of squarefree monomial ideals, then, viewing $I_n^\vee$  as an ideal of $R_{n+1}$, we have $I_{n+1}^\vee\subseteq I_n^\vee \cdot R_{n+1}$. 
\end{cor}

The following result describes the behavior of the Stanley-Reisner complexes of a $\Sym$-invariant chain of squarefree monomial ideals.
\begin{lemma}\label{lem: subcomplex}
Let $(I_n)_{n\in \mathbb{N}}$ be a $\Sym$-invariant chain of squarefree monomial ideals. Then 
 $\Delta(I_n)$ is a subcomplex of $\Delta(I_{n+1})$ for $n\gg 0$.
\end{lemma}

\begin{proof}
Let $n\gg 0$ such that the chain has stabilized, i.e., $I_{n+1} = \Sym(n+1)(I_n)$. 
Let $F$ be a face of $\Delta(I_n)$, i.e., $F\notin I_n$, where we identify faces with monomials. If, by contradiction, $F\notin \Delta(I_{n+1})$, then $F\in I_{n+1}=\Sym(n+1)(I_n)$. As $F\in \Delta(I_n)$, it follows that $F\in\Sym(n)(I_n)$, which yields a contradiction.
\end{proof}

These results lead us to the following questions, which the remainder of this paper will be dedicated to answering. Recall that any $\Sym (\infty)$-invariant  monomial ideal of $R$ is generated by the $\Sym(\infty)$-orbits of finitely many monomials. 

\begin{question}\label{question: mainQuestion} 
Consider a $\Sym$-invariant chain $(I_n)_{n\in\NN}$ of squarefree monomial ideals. What are the monomials whose $\Sym(n)$-orbits minimally generate  $I_n^\vee$ for large $n$ in terms of the finitely many monomials whose $\Sym(\infty)$-orbits generate the colimit $\lim I_n$? How does the number of orbit generators of $I_n^{\vee}$ grow as $n\ra \infty$?
\end{question}

\subsection{Invariant Chains of Simplicial Complexes} 

By Stanley-Reisner theory, \Cref{question: mainQuestion} can also be stated in terms of simplicial complexes. We make this explicit now. 

\begin{definition}
         \label{def:invariant chain complexes}
Consider a sequence $(\Delta_n)_{n \in \NN}$  of 
simplicial complexes $\Delta_n$, where $\Delta_n$ has vertex set $V_n = [n]^c\coloneqq\{j_i \; \mid  \; i \in [c], j \in [n]\}$. 
Define an action of $\Sym (n)$  on $V_n$ by $\sigma \cdot j_i = (\sigma(j))_i$. The sequence $(\Delta_n)_{n \in \NN}$ is called a  \emph{$\Sym$-invariant chain} if $\Sym( n)$ maps any  face of $\Delta_m$ onto a face of $\Delta_n$ whenever $m \le n$ and the $c$-fold cone over $\Delta_n$ with apices $V_{n+1} - V_n$ is contained in $\Delta_{n+1}$. 
\end{definition}

As a consequence of the previous definition, every simplicial complex in a $\Sym$-invariant chain is $\Sym$-invariant as defined in the previous subsection. As mentioned in the introduction, any $\Sym$-invariant chain of simplicial complexes stabilizes. 

\begin{prop}
   \label{prop:stabilization of chain of complexes}
If $(\Delta_n)_{n \in \NN}$ is a $\Sym$-invariant chain of simplicial complexes, then there is some integer $n_0$ such that $\Sym(n+1)$-orbit of the cone over $\Delta_n$ with apices $V_{n+1} -V_n$ equals $\Delta_{n+1}$ whenever $n \ge n_0$
\end{prop}

\begin{proof}
Identify any vertex $j_i \in V_n$ and the variable $x_{i, j} \in R_n$. Since each simplicial complex $\Delta_n$ is $\Sym (n)$-invariant, \Cref{prop: symmetricFaces} combined with \Cref{prop:alexDualComplex} implies that the Alexander duals $I_{\Delta_n^\vee}=(I_{\Delta_n})^\vee$ are $\Sym(n)$-invariant. The assumption that the cone over 
$\Delta_n$ with apices $V_{n+1} - V_n$ is contained in $\Delta_{n+1}$ implies that the extension ideal $I_{\Delta_n} \cdot R_{n+1}$ contains $I_{\Delta_{n+1}}$ and by \Cref{cor: AD reverse chain} we get that $I_{\Delta_n^\vee}R_{n+1}\subseteq I_{\Delta_n^\vee}$. Hence, the sequence $(I_{\Delta_n^\vee})_{n \in \NN}$ is $\Sym$-invariant. 
It follows from results in \cite{C, HS} or \cite{NR2} that the sequence stabilizes, that is, there is some integer $n_0$ such that $I_{\Delta_n^\vee} \cdot R_{n+1} = I_{\Delta_{n+1}^\vee}$ for $n\geq n_0$. The required stabilization for $(\Delta_n)_{n\in \NN}$ follows.
\end{proof}

This result shows that \Cref{question: mainQuestion} is equivalent to the problem of describing the facets of $\Delta_n^{\vee}$ and the number of $\Sym (n)$-orbits of facets of $\Delta_n^{\vee}$ for $n \gg 0$ if $(\Delta_n)_{n \in \NN}$ is a $\Sym$-invariant chain of simplicial complexes.

\section{Alexander Duals of $\Sym$-invariant Chains and Order Ideals}  
\label{sec: alexDualsOrderIdeals}
We introduce some notation specific to the case of squarefree monomial orbit generators that we will use throughout the rest of the paper. With this notation in hand, we can translate the question of determining whether or not a certain monomial belongs to the Alexander dual of a $\Sym$-invariant ideal into a combinatorial question closely related to the results from Section \ref{sec: avoidanceUpToSym}.

\begin{definition}
Let $\mm$ be a squarefree monomial in $R_n = \mathbf{k}[x_{i,j} ~\mid~ i\in [c], j\in [n] ]$ for a fixed $c,n$. Then the \emph{exponent matrix} associated to $\mm$ is the $c\times n$ matrix \([b_{i,j}]\in \{0,1\}^{c\times n}\), where \(b_{i,j}=1\) if \(x_{i,j}\) divides $\mm$ and $b_{i,j}=0$ otherwise. If $B$ is the exponent matrix of $\mm$, we write $\mm = \xx^B$.
\end{definition}

\begin{example}\label{ex: exponent matrix of a monomial}
For \(c=3\), the monomial \(x_{1,1}x_{1,4} x_{1,5} x_{2,1}x_{2,2}x_{2,5} x_{3,1}x_{3,2}x_{3,5} \in R_5\) has exponent matrix 
\[
B = \begin{bmatrix}
1 & 0 & 0 & 1 & 1\\
1 & 1 & 0 & 0 & 1\\
1 & 1 & 0 & 0 & 1
\end{bmatrix}.
\]
\end{example}

Given an orbit generator $\xx^B$ of a $\Sym$-invariant squarefree monomial ideal $I_n$ in $R_n$ and a permutation $\sigma\in \Sym(n)$, denote by $\sigma(B)$ the matrix obtained from $B$ by applying the permutation $\sigma$ to its set of column vectors. Since any monomial 
$\sigma(\xx^B) = \xx^{\sigma(B)}$ also could have been chosen as an orbit generator of $I_n$, we would like to have a 
``standard'' representative for monomial orbit generators to avoid ambiguity.

\begin{notation}\label{not: standardMatrix}
For a subset \(S\subseteq [c]\) and an integer \(k_S \geq 0\), let \(\ds A_S(k_S)\in \{0,1\}^{c\times k_S}\) denote the exponent matrix of the monomial \(\ds\prod_{i\in S}\prod_{j\in [k_S]}x_{i,j}\), i.e., the \((i,j)\) entry of $A_S(k_S)$ with $k_S > 0$ is \(1\) if and only if \(i\in S\) and it is $0$, otherwise.

Order the subsets of \([c]\) with a lexicographic ordering induced by the ordering \(1>2>\cdots >c\) on \([c]\).
That is, for subsets \(S\neq T\) of \([c]\), we let \(S>T\) if 
\begin{itemize}
\item \(|S|>|T|\) or, 
\item $\abs{S} = \abs{T}$ and \(s_i>t_i\), where $S=\{s_1<s_2<\cdots<s_k\}$, $T=\{t_1<t_2<\cdots<t_k\}$ and $i=\min(\ell~\mid~s_\ell\neq t_\ell)$.
\end{itemize}
\end{notation}

\begin{definition}\label{def: standard matrix}
For any monomial \(\xx^B\in R_n\), we call the unique matrix 
\[
A=\left[A_{S_1}(k_{S_1})\;\bigg\vert\; A_{S_2}(k_{S_2})\;\bigg\vert\;\cdots\;\bigg\vert\; A_{S_{2^c-1}}(k_{S_{2^c-1}})\right]
\]
with \(\ds [c]=S_1>S_2>\cdots >S_{2^c-1}=\{c\}\) and suitable integers \(k_{S_i}\geq 0\) such that \(\xx^A\in \Sym(n)\cdot \xx^B\) the \emph{standard matrix} of \(\xx^B\). 
Using the ordering on the subsets from Notation \ref{not: standardMatrix}, we often simply write 
\[
A=\left[\Ast_{ \emptyset\neq T \subseteq [c]}A_T(k_T)\right]
\]
to denote the standard matrix.
\end{definition}

\begin{remark} 
    \label{rem:padding}
If $A$ is a standard matrix with $k<n$ columns, then it is understood that the monomial $\xx^A\in R_n$ refers to the monomial with exponent matrix $A$ padded with $n-k$ columns of zeroes at the end.
\end{remark}

\begin{example}\label{ex: standard matrix}
For \(c=3\), the monomial \(\xx^B\) in Example~\ref{ex: exponent matrix of a monomial} has standard matrix 
\[
A=\left[A_{[3]}(2)\;\bigg\vert\;A_{\{2,3\}}(1) \;\bigg\vert\; A_{\{1\}}(1)\right]=\begin{bmatrix}1&1&0&1\\1&1&1&0\\1&1&1&0\end{bmatrix}
\]
where \(k_S = 0\) for all other subsets.
\end{example}

\begin{definition}\label{def:support}
Let $A\in \{0,1\}^{c\times n}$ be a $0-1$-matrix with columns $A_1,\ldots,A_n$. 
Define the \textit{support} of a column $A_i$, denoted \(\Supp(A_i)\), to be the set of row indices for which the column \(A_i\) has nonzero entries. We also say that $A_i$ has \textit{support} in $S \subseteq [c]$ or \emph{is supported} in $S$ if $S=\Supp(A_i)$.

 The \emph{support} of the matrix \(A\) is the set 
 \[
 \Supp (A) = \{T \subseteq [c] \; \mid \;  T = \Supp (C) \text{ for some column $C$ of $A$}\}. 
 \]
\end{definition}
We note that matrices $A$ and $\sigma(A)$ have the same support set for any $\sigma \in \Sym(n)$. 
With this notation established, we now apply Theorem \ref{thm: combinatorial statement} to study Question \ref{question: mainQuestion} in the case where ideal $I_n \subset R_n$ is generated by a single squarefree monomial orbit $\Sym (n) \cdot \xx^A$ for $n \gg 0$. We characterize when a monomial $\xx^B\in R_n$ is in the Alexander dual of $I_n$ for a fixed $n$. In this case, the image of the map $f$ in the hypotheses of Theorem \ref{thm: combinatorial statement} will correspond to the support set of the standard matrix of $\xx^A$, while the image of $g$ will be given by the support set of the standard matrix of $\xx^B$.

\begin{definition}\label{not: OneOrbitLemmas}
Given a matrix $A \in \{0,1\}^{c\times m}$ with
$$ A = \newMatrix{2^{[c]}}{k}, \quad  k_T \geq 0,$$
the $\Sym$-invariant chain   \emph{generated by $\xx^A$}  is the chain of ideals $(I_n)_{n \in \NN}$  with   $I_n =\langle \Sym(n) \cdot \xx^A \rangle \subset R_n$ if $n \ge m$ and $I_n = 0$ otherwise. 
\end{definition}

\begin{prop}[Alexander Dual Membership Criterion] \label{coro: inequalities for alexander dual} 
Let $(I_n)_{n \in \NN}$ be a $\Sym$-invariant chain of ideals generated by $\xx^A$ with $A = \newMatrix{2^{[c]}}{k} \in \{0,1\}^{c\times m}$.  
Fix an integer $n \geq m$ and consider a matrix $B = \newMatrix{2^{[c]}}{\ell}\in\{0,1\}^{c\times n}$. Then the following conditions are equivalent:
\begin{enumerate}[(a)]
    \item The monomial $\xx^B$ does not lie in $I_n^\vee$.
    \item For every order ideal $J$ of $2^{[c]}$, one has $\sum_{T\in J} k_T \leq \sum_{T\in J}\ell_{T^C}$.
    \item For every order ideal $J$ of $2^{[c]}$, one has $\sum_{T\in J} \ell_T \leq \sum_{T\in J}k_{T^C}$.
\end{enumerate}
\end{prop}

\begin{proof}
Define maps $f,g \colon [n] \to 2^{[c]}$ by
$$f(i) = \Supp(A_i), \quad \text{and} \quad g(i) = \Supp(B_i).$$ 
\Cref{fact:mon in dual} implies that  $\xx^B \notin I_n^\vee$ if and only if there exists a permutation $\sigma \in \Sym(n)$ such that $f(i) \cap g(\sigma(i))=\varnothing$ for every $i\in[n]$. By Theorem \ref{thm: combinatorial statement}, this is equivalent to conditions $(b)$ and $(c)$, where $k_T = |f^{-1}(T)|$ and $\ell_T = |g^{-1}(T)|$ for each $T\subseteq [c]$.
\end{proof}


\begin{definition}\label{not: divUpToSym}
Consider two matrices $B,B'\in \{0,1\}^{c\times n}$. We say that $\xx^{B'}$ \textit{divides $\xx^B$ up to symmetry} if there exists $\sigma\in\Sym(n)$ such that $\xx^{B'}$ divides $\xx^{\sigma\cdot B}$.
\end{definition}

We will frequently say that $\xx^{B'}$ divides $\xx^B$ but mean that this is up to symmetry.

The following criterion for determining divisibility up to symmetry is also a consequence of Theorem \ref{thm: combinatorial statement}. It will be critical for constructing a minimal generating set for $I_n^\vee$ in both the one-orbit and the general case.

\begin{lemma}[Divisibility Lemma]\label{lem: divisibility up to symmetry}
Consider two matrices $$B = \newMatrix{2^{[c]}}{\ell} \quad \text{and} \quad B' = \newMatrix{2^{[c]}}{j}$$ in $\{0,1\}^{c\times n}$. 
Then $\xx^{B'}$ divides $\xx^B$ up to symmetry if and only if for every order ideal $J$ of $2^{[c]}$ that is generated by a non-empty  subset of $\Supp (B')$, one has
$$\sum_{T\in J}j_T \leq \sum_{T\in J}\ell_T.$$
\end{lemma}

\begin{proof}
By definition, $\xx^{B'}$ divides $\xx^B$ up to symmetry if and only if, up to a permutation of columns, one has $\Supp(B'_i) \subseteq \Supp(B_i)$ for each column $i$. Define $\overline{B}$ to be the matrix whose column supports are the complements of those of $B$; that is, $\Supp(\overline{B}_i) = \Supp(B_i)^C$ for each $i \in [n]$. Then $\xx^{B'}$ divides $\xx^B$ up to symmetry if and only if $\Supp(\overline{B}_i) \cap \Supp(B_i') = \varnothing$ for all $i$. 
  By  \Cref{thm: combinatorial statement}, this equality is true for some permutation of the columns of $\overline{B}$ exactly if
\begin{equation}
   \label{eq:divisibility inequality}
\sum_{T\in J} j_T \leq \sum_{T\in J}\ell_{{(T^C)}^C} = \sum_{T\in J} \ell_T
\end{equation}
for every order ideal $J$ of $2^{[c]}$. 

It remains to show that is suffices to check Inequality \eqref{eq:divisibility inequality} for order ideals generated by non-empty subsets of $\Supp (B')$. To this end consider any order ideal $I$ of $2^{[c]}$ and assume that Inequality \eqref{eq:divisibility inequality} is true for $\langle E\rangle$, where $E = I \cap \Supp (B')$. Then,

\[
\sum_{T\in I} j_T = \sum_{T\in \langle E\rangle} j_T\leq \sum_{T\in \langle E\rangle}\ell_T\leq \sum_{T\in I}\ell_T
\] 
which shows Inequality \eqref{eq:divisibility inequality} for $I$. This completes the argument. 
\end{proof} 

\begin{remark}
    \label{rem:divisibility} 
For $J = 2^{[c]}$, the condition $\sum_{T\in J}j_T \leq \sum_{T\in J}\ell_T$ is automatically true with equality as both sides equal $n$, the number of columns of $B'$ and of $B$. Thus, in \Cref{lem: divisibility up to symmetry}, it is enough to check the stated condition for order ideals $J \neq 2^{[c]}$. 
\end{remark}

We close this section with another result that we will use repeatedly to simplify our arguments and to reduce the number of possible monomials under consideration for the minimal orbit generators of $I_n^\vee$.

\begin{lemma}[Exchange Lemma]\label{lem: exchangeLemma} 
Adopt the notation of \Cref{coro: inequalities for alexander dual}. Let $J$ be an order ideal of $2^{[c]}$. 
Assume  that $\xx^B$ is a minimal generator of $I_n^\vee$ satisfying 
$$
\sum_{T\in J} \ell_T > \sum_{T\in J} k_{T^C}.
$$
Then the following hold:
\begin{itemize}
\item[(a)] If $\ell_S>0$ for some $S\in J$, then $S$ is a minimal element of $J$.
\item[(b)] $\ell_U=0$ for all $\emptyset\neq U\in 2^{[c]}-J$ and $\ell_\emptyset=\sum_{T\in 2^{[c]}-\overline{J}} k_{T}-1=n-1-\sum_{T\in J} k_{T^C}$.
\end{itemize}

\end{lemma}

\begin{proof} 
For (a), assume by contradiction that there are elements $S' \subsetneqq S$ in $J$ with $\ell_S>0$. 
Consider the matrix
$$
B' = \left[A_{S'}(\ell_{S'}+1)\right]\ast\left[A_S(\ell_S-1)\right]\ast\left[\Ast_{U\in 2^{[c]} - \{S,S'\}} A_U(\ell_U)\right].
$$
It differs from $B$ in only two columns. Since $S' \subsetneqq S$, we know that $\xx^{B'}$ is a proper divisor of $\xx^B$. Thus, it suffices to show that $\xx^{B'}\in I_n^\vee$. This follows by the Alexander Dual Membership Criterion from the fact that the sums of the number of columns of $B$ and $B'$ that are supported in $S$ or $S'$ are the same. 

For (b), we let $a \coloneqq n-\sum_{T\in J} \ell_{T}$. Setting $\tilde{B}=[\Ast_{T\in J}A_T(\ell_T)\ast A_{\emptyset}(a)]$, the Alexander Dual Membership Criterion directly implies that $\xx^{\tilde{B}}\in I_n^\vee$. Since $\xx^{\tilde{B}}$ divides $\xx^B$ and the latter is a minimal generator, we must have $\xx^B=\xx^{\tilde{B}}$ and hence $\ell_U=0$ for all $\emptyset\neq U\in 2^{[c]}-J$. Moreover, we claim that 
$$
\sum_{T\in J} \ell_T = \sum_{T\in J} k_{T^C}+1.
$$
Otherwise, assume that $\sum_{T\in J} \ell_T = \sum_{T\in J} k_{T^C}+c$ for some $c\in \mathbb{N}$ with $c>1$. Denote by $\overline{B}$ a matrix obtained from $B$ by deleting $c-1$ columns with support in some element of $J$ and adding $c-1$ columns of all zeroes. The Alexander Dual Membership Criterion implies that $\xx^{\overline{B}}\in I_n^\vee$. As $\xx^{\overline{B}}$ properly divides $\xx^B$ by construction, we arrive at a contradiction and the claim is shown. 
It follows that 
$$
\ell_\emptyset=n-\sum_{T\in J}\ell_T=n-\sum_{T\in J}k_{T^C}-1=\sum_{T\in 2^{[c]}-\overline{J}}k_T-1,
$$
which finishes the proof.
\end{proof}

\section{The Case of One Orbit}\label{sec: oneOrbit}

Aided by the results and setting established in Section \ref{sec: alexDualsOrderIdeals}, we produce the generating sets for the Alexander duals of $\Sym$-invariant chains of Stanley-Reisner ideals in the case where they have a single generator up to symmetry. We treat this case separately since there the ideas are more transparent and will be helpful for proving the general case. 

\begin{notation}
Throughout this section, fix a nonzero matrix $A \in \{0,1\}^{c\times m}$ given by
$$ A = \newMatrix{2^{[c]}}{k}, \quad  k_T \geq 0, $$
and consider the $\Sym$-invariant chain of ideals $(I_n)_{n \in \NN}$ generated by $\xx^A$.  Thus, $I_n =\langle \Sym(n) \cdot \xx^A \rangle \subset R_n = \mathbf{k}[x_{i,j} \mid i\in [c], j\in [n] ]$ if $n \ge m$ (see  \Cref{not: OneOrbitLemmas}). 
\end{notation}

Below we outline the basic construction for obtaining a generating set for $I_n^\vee$.

\begin{construction}\label{const: GC_(n)}
Let $n \ge m$. By adding $n-m$ columns of zeroes to $A$, we view $A$ as a $(c\times n)$-matrix. Consider a matrix 
$B = \newMatrix{2^{[c]}}{\ell}\in\{0,1\}^{c\times n}$. Thus, $n = \sum_{T \in 2^{[c]}}k_T = \sum_{T \in 2^{[c]}}\ell_{T^C}$. Hence \Cref{coro: inequalities for alexander dual} yields that  
the monomial $\xx^B$ lies in $I_n^\vee$ if and only if there exists a proper order ideal $J\subsetneqq 2^{[c]}$ such that 
\begin{equation}\label{eq:witness}
\sum_{T\in J}\ell_T > \sum_{T\in J}k_{T^C}.
\end{equation}
If, in addition, $\xx^B$ is a minimal generator of $I_n^\vee$, then by the Exchange Lemma we have $\Supp(B)\subseteq \cC\cup\{\emptyset\}$, where $\cC$ is the antichain generating $J$ as an order ideal. Moreover, up to padding with zero columns (see \Cref{rem:padding}), $B$ has to lie in the set
$$\cG_\cC(n) := \left \{\newMatrix{\cC}{j} \; \mid \; j_T \geq 0, \sum_{T \in \cC} j_T = n+1-k_\cC\right\},$$
where 
\begin{equation} 
    \label{eq:def of kC}
k_\cC := \sum_{T\in 2^{[c]} - \overline{\langle \cC \rangle}}k_T.
\end{equation}
In the proof of \Cref{lem: equivalence of inequalities}, we showed that \eqref{eq:witness} is equivalent to
\begin{equation*}
\sum_{T\in 2^{[c]}-\overline{J}}k_T > \sum_{T\in 2^{[c]}-\overline{J}}\ell_{T^C}\geq 0.
\end{equation*}
It follows that $k_\cC \geq 1$.

\end{construction}

Since non-empty antichains are in one-to-one correspondence with proper order ideals, the above construction yields the following generating set, which is not necessarily minimal. 

\begin{prop}\label{prop: redundant gen set one orbit}
For any $n \ge m$, the Alexander dual $I_n^\vee = \langle \Sym(n) \cdot \xx^A\rangle^\vee$ is generated by the set
$$\bigcup_{\substack{\text{ $\cC$ antichain}\\  \varnothing \neq \cC \subsetneqq 2^{[c]}, \ k_\cC \geq 1}} \left \{ \xx^B \; \mid \;  B \in \cG_\cC(n) \right\}.$$
\end{prop}

Observe that the condition $k_\cC \geq 1$ implies that the order ideal generated by $\cC$ is not equal to $2^{[c]}$. 

We discuss \Cref {prop: redundant gen set one orbit} in an example. The example also suggests subsequent refinements to obtain a minimal generating set.

\begin{example}\label{ex: one-orbit}
Consider the matrix
$$
A = \begin{pmatrix}
      1&1&0\\
      1&0&1\\
      0&1&1\end{pmatrix}
$$
and set $I_4 = \langle \Sym(4) \cdot \xx^A\rangle$.
First, we illustrate Construction \ref{const: GC_(n)} to describe the set of generators of $I_4^\vee$ in $\cG_\cC(4)$ for the antichain $\cC = \{\{2\},\{3\}\}$ of  $2^{[3]}$. To simplify notation, we will often omit parentheses around sets and commas between elements of sets, e.g., we write $k_{123}$ instead of $k_{\{1,2,3\}}$. Notice that 
$$ \begin{tikzpicture}
    \node (top) at (0,0) {$\{1,2,3\}$}; \node (c) at (-2.5,-1) {$\langle \cC \rangle = $};
    \node (12) at (-1,-1) {$\{1,2\}$};
    \node (23) at (0,-1) {$\{2,3\}$};
    \node (13) at (1,-1) {$\{1,3\}$};
    \node (2) at (-0.5,-2) {$\{2\}$};
    \node (3) at (0.5, -2) {$\{3\}$};
    \draw (top)--(12)--(2)--(23)--(top)--(13)--(3)--(23);
    \node (top2) at (6,-0.5) {$\{1,2,3\}$};
    \node (comp) at (6,-1.5) {$\{2,3\}$};
    \node (c2) at (3.5, -1) {\text{and} \quad $2^{[3]} - \overline{\langle \cC\rangle} = $};
    \draw (top2)--(comp);
\end{tikzpicture}
$$
so $
k_\cC = \sum_{T\in 2^{[c]}-\overline{\langle\cC\rangle}} k_T = k_{123}+k_{23} = 0+1 = 1$, yielding the equality
\begin{equation}\label{eq: exampleOneOrbit}
j_{2} + j_{3} = n+1 - k_\cC = 4 + 1 - 1 = 4. 
\end{equation}
Below we list all nonnegative integer solutions to Equation \ref{eq: exampleOneOrbit} and the corresponding matrix for each solution.
\begin{align*}
    j_{2} = 4, \quad j_{3} = 0 \qquad &\longleftrightarrow \qquad A_2(4) \\
    j_{2} = 3, \quad j_{3} = 1 \qquad &\longleftrightarrow \qquad \left[A_2(3) \ast A_3(1)\right] \\
    j_{2} = 2, \quad j_{3} = 2 \qquad &\longleftrightarrow \qquad \left[A_2(2) \ast A_3(2)\right]\\
    j_{2} = 1, \quad j_{3} = 3 \qquad &\longleftrightarrow \qquad \left[A_2(1) \ast A_3(3)\right]\\
    j_{2} = 0, \quad j_{3} = 4 \qquad &\longleftrightarrow \qquad A_3(4).
\end{align*}
Second, we show that not all of these generators are \textit{minimal} generators of $I_4^\vee$. To see this, consider the two nonempty antichains properly contained in $\cC$ corresponding to $\cC_1 = \{2\}$ and $\cC_2 = \{3\}$, which have corresponding equations $j_2 = 3$ and $j_3 = 3$, respectively.
By the Divisibility Lemma \ref{lem: divisibility up to symmetry}, the monomial with exponent matrix $A_{2}(3)$ corresponding to the first of these equations will divide the monomials corresponding to the matrices $A_{2}(4)$ and $A_{2}(3)\ast A_{3}(1)$, which arise from the first two solutions of Equation (\ref{eq: exampleOneOrbit}). Similarly, the monomial with exponent matrix $A_{3}(3)$ from the second equation above will divide the monomials with exponent matrices $A_{2}(1)\ast A_{3}(3)$ and $A_{3}(4)$ corresponding to the last two solutions of Equation (\ref{eq: exampleOneOrbit}). It follows that the only minimal generator of $I_n^\vee$ arising from a solution of Equation (\ref{eq: exampleOneOrbit}) has exponent matrix
\begin{equation}\label{eq: exampleOneOrbitMatrix}
A_{2}(2) \ast A_{3}(2) = \begin{bmatrix} 0 & 0 & 0 & 0 \\ 0 & 0 & 1 & 1 \\ 1 & 1 & 0 & 0
\end{bmatrix}. 
\end{equation}
\end{example}

Example \ref{ex: one-orbit} indicates that in order to cut down $\bigcup \cG_\cC (n)$ to a \textit{minimal} generating set, we need to ignore elements of  $\cG_\cC (n)$ which are overridden by elements coming from antichains in the \emph{lower} order ideal $\lowerIdeal{\cC}$. This is true in general because of the following consequence of the Divisibility Lemma \ref{lem: divisibility up to symmetry}. 

\begin{cor} 
   \label{cor:from dibisibility lemma} 
For any  antichain $\cC$ of $2^{[c]}$, one has: 
\begin{itemize}
    \item[(a)]  No two distinct elements of $\cG_\cC(n)$ divide each other.
    \item[(b)]  If $B \in \cG_\cC(n)$ and $B' \in \cG_{\cC'}(n)$ for some antichain $\cC'$ such that $\xx^{B'}$ divides $\xx^B$ up to symmetry, then $\Supp (B')  \subseteq \lowerIdeal{\cC}$.
\end{itemize}
\end{cor} 

As further preparation, we establish the following useful criterion. 

\begin{lemma}
   \label{lem:orbit intersect one-orbit case}
For nonempty antichains $\cC, \cC'$ of $2^{[c]}$ satisfying $\{ \varnothing \} \neq \cC' \subset \lowerIdeal{\cC}$ and some integer $\lambda \ge 0$, consider a matrix $B = \newMatrix{\cC}{j}$ with $j_T \in \NN_0$ and a set
\[
\Lambda = \big \{(i_S)_{S \in \cC'} \; \mid \; \sum_{S \in \cC'} i_S = \lambda, i_S \in \NN_0 \big \}. 
\]
Then there is some $(\ell_S)_{S \in \cC'} \in \Lambda$ such that $\xx^{B'}$ with $B' = [\Ast_{S \in \cC'} A_S (\ell_S)]$ divides $\xx^B$ up to symmetry if and only if 
\[
\lambda \le \sum_{T \in \cC \cap \langle \cC' \rangle} j_T. 
\]
\end{lemma} 

\begin{proof}
If $\xx^{B'}$ divides $\xx^B$ then, applying the Divisibility Lemma \ref{lem: divisibility up to symmetry} to the order ideal $J = \langle \cC'\rangle$, we obtain
\[
\lambda = \sum_{S \in \langle \cC'\rangle} \ell_S \le \sum_{T \in \langle \cC' \rangle} j_T = \sum_{T\in \langle \cC'\rangle \cap \cC}j_T,
\]
as claimed. 

We show the converse by using induction on the cardinality of $\cC'$. If $|\cC'| = 1$ with $\cC' = \{ S\}$, then $\Lambda=\{\lambda\}$. Since$\lambda \le \sum_{T\in \langle S \rangle \cap \cC}$ by assumption, the monomial $\xx^{B'}$, with $B'=A_S(\lambda)$ divides $\xx^B$. 

Assume $|\cC'| \ge 2$. Choose any $S \in \cC$ and set 
\[
\ell_S := \min \big \{ \lambda, \ \sum_{T \in (\cC - \langle S \rangle) - (\cC \cap \langle \cC' - S \rangle)} j_T \big \}.
\]
Since $\lambda - \ell_S \ge 0$, the induction hypothesis applied to $\cC' - S$ shows there is some $(\ell_T)_{T \in \cC' - S} \in \NN_0^{|\cC'| - 1}$ such that $\sum_{T \in \cC' - S} \ell_T = \lambda - \ell_S$ and $\xx^{\tilde{B}}$ with $\tilde{B} = [\Ast_{T \in \cC' - S} A_T (\ell_T)]$ divides $\xx^B$. By the Divisibility Lemma \ref{lem: divisibility up to symmetry}, this means that for every subset $\cD \neq \varnothing$ of $\cC' - S$, one has
\begin{align}
    \label{eq:ineq induction}
\sum_{T \in \cD} \ell_T \le \sum_{T \in \cC \cap \langle \cD \rangle} j_T. 
\end{align}

We claim that $\xx^{B'}$ with $B'=\tilde{B}\ast [A_S(\ell_S)]$ divides $\xx^B$ up to symmetry.
We need to show that for every subset $\cD \neq \varnothing$ of $\cC'$, one has $\sum_{T \in \cD} \ell_T \le \sum_{T \in \cC \cap \langle \cD \rangle} j_T. $ Indeed, if $S$ is not in $\cD$, this is true by Inequality \eqref{eq:ineq induction}. If $S$ is in $\cD$ then $\cC \cap \langle \cD-S \rangle \subseteq \cC \cap \langle \cC' - S \rangle$, and so 
\begin{align}
\ell_S \le \sum_{T \in (\cC \cap \langle S \rangle) - (\cC \cap \langle \cC' - S \rangle)} j_T & \le \sum_{T \in (\cC \cap \langle S \rangle) - (\cC \cap \langle \cD - S \rangle)} j_T. 
\end{align}
Using Inequality \eqref{eq:ineq induction}, we know that 
\[
\sum_{T \in \cD - S} \ell_T \le \sum_{T \in \cC \cap \langle \cD -S \rangle} j_T.
\]
Adding the last two inequalities, we obtain 
\begin{align*}
\sum_{T \in \cD} \ell_T  & \le \sum_{T \in \cC \cap \langle \cD -S \rangle} j_T +  
 \sum_{T \in (\cC \cap \langle S \rangle) - (\cC \cap \langle \cD - S \rangle)} j_T \\
& = \sum_{T \in \cC \cap \langle \cD \rangle} j_T
\end{align*}
because $\cC \cap \langle \cD \rangle$ is the disjoint union of $\cC \cap \langle \cD -S \rangle$ and $(\cC \cap \langle S \rangle) - (\cC \cap \langle \cD - S \rangle)$. 
\end{proof}

We now apply this result to our particular generating set from Proposition \ref{prop: redundant gen set one orbit} to find  a minimal generating set. For stating the result,  we make the following definition. 
Recall that $k_{\cC} = \sum_{T\in 2^{[c]} - \overline{\langle \cC \rangle}}k_T$ (see \Cref{eq:def of kC}).

\begin{definition}\label{def: minGenSet 1 orbit}
For any antichain $\cC\subset 2^{[c]}$ with $k_{\cC}\geq 1$, set 
\begin{align}\label{eqn: minGenSet 1 orbit}
M\cG_{\cC}(n) := & \left \{ \newMatrix{\cC}{j} \;  \mid \;  j_T \geq 0, \sum_{T\in\cC} j_T = n+1-k_{\cC} \ \text{ and }   \right. \\
& \; \; \left. \sum_{T\in \cC- \ideal{\cC'}}j_T > k_{\cC'}-k_{\cC} \text{ for any antichain } \cC'\subseteq \lowerIdeal{\cC} - \{\varnothing, \cC \} \right\}.   \nonumber
\end{align}
\end{definition}

\begin{theorem}\label{cor: one orbit minimal gen set}
If $n \ge \max\{k_\cC-1 : \cC \text{ an antichain in }2^{[c]} \text{ with }  k_{\cC} \ge 1\}$, then 
the Alexander dual $I_n^\vee$ of $I_n = \ideal{\Sym(n) \cdot \xx^A}$ is minimally generated by the $\Sym (n)$-orbits of the monomials in the  set
$$
\cG (n) \coloneqq \bigcup_{\substack{\cC \text{ antichain} \\ \varnothing \neq \cC \subset 2^{[c]},  k_\cC \geq 1}} \left\{\xx^B ~\mid~ B \in M\cG_{\cC}(n)\right\}.$$ 
\end{theorem}

\begin{proof} We proceed in several steps. 
\smallskip 

(I) Observe that for any element $B = \newMatrix{\cC}{j}\in M\cG_{\cC}(n)$ and any antichain $\cC'$, one has 
$$
n+1-k_\cC = \sum_{T\in \cC}j_T = \sum_{T\in \cC\cap \ideal{\cC'}}j_T + \sum_{T\in \cC- \ideal{\cC'}}j_T. 
$$
Thus, the condition 
\[
n+1-k_{\cC'} \leq \sum_{T\in\cC \cap \langle \cC' \rangle}j_T 
\]
is equivalent to 
\[
\sum_{T\in \cC- \ideal{\cC'}}j_T \le k_{\cC'}-k_{\cC}. 
\]

(II) We now show that $\cG (n)$ generates the ideal $I_n^{\vee}$. Using the generating set given in  \Cref{prop: redundant gen set one orbit}, it suffices to prove that any monomial $\xx^B$ with $B \in \cG_{\cC} (n) - M\cG_{\cC} (n)$ is divisible by a monomial in $\cG (n)$. By Step (I), the fact that $B$ is not in  $M\cG_{\cC} (n)$ implies the existence of an antichain  $\cC_1\subseteq \lowerIdeal{\cC} - \{\varnothing, \cC \}$ satisfying 
\[
n+1-k_{\cC_1} \leq \sum_{T\in\cC \cap \langle \cC_1 \rangle}j_T.  
\] 
Hence \Cref{lem:orbit intersect one-orbit case} shows that there is a matrix $B_1 = [\Ast_{S \in \cC_1} A_S (\ell_S)] \in \cG_{\cC_1} (n)$ such that $\xx^{B_1}$ divides $\xx^B$. Note that $\cC_1 \neq \cC$ implies $\lowerIdeal{\cC_1} \subsetneqq \lowerIdeal{\cC}$. 

If $B_1$ is not in $M\cG_{\cC_1} (n)$,  repeating the above argument gives the existence  of a matrix $B_2 = [\Ast_{S \in \cC_2} A_S (\ell_S)] \in \cG_{\cC_2} (n)$ for some antichain  $\cC_2\subseteq \lowerIdeal{\cC_1} - \{\varnothing, \cC_1 \}$   such that $\xx^{B_2}$ divides $\xx^{B_1}$. Moreover, we have strict inclusions
\[
\lowerIdeal{\cC_2} \subsetneqq \lowerIdeal{\cC_1} \subsetneqq \lowerIdeal{\cC}. 
\]
Since any such decreasing chain is finite, repeating the above argument sufficiently often we eventually get an antichain $\tilde{\cC} \neq \varnothing$ and some $\tilde{B} \in M\cG_{\tilde{\cC}} (n)$ such that $\xx^{\tilde{B}}$ divides $\xx^B$. Since 
$\xx^{\tilde{B}}$ is in $\cG (n)$, this proves that $\cG (n)$ is a generating set of $I_n^{\vee}$. 
\smallskip 

(III) Finally, we show that $\cG (n)$ minimally generates $I_n^{\vee}$, that is, $\cG (n)$ does not contain any redundant monomials. 

Suppose there are monomial $\xx^B  \neq \xx^{B'}$ in  $\cG (n)$ with $B = [\Ast_{T \in \cC} A_S (j_T)] \in M\cG_{\cC} (n)$ and $B' =  [\Ast_{T \in \tilde{\cC}} A_S (i_T )]\in M\cG_{\tilde{\cC}} (n)$ for some nonempty antichains  $\cC, \tilde{\cC}$ with $k_{\cC} \ge 1$ and $k_{\tilde{\cC}} \ge 1$ such that  $\xx^{B'}$ divides $\xx^{B}$. 
Hence,  $\cC' : = \Supp (B')$ is a nonempty antichain,  and $\cC' \subseteq \tilde{\cC}$ implies $k_{\cC'} \ge k_{\tilde{\cC}}  \ge 1$. 
Since $\xx^{B'}$ divides $\xx^B$ we obtain from the Divisibility Lemma \ref{lem: divisibility up to symmetry} applied with 
$J = \langle \cC' \rangle$  that  $\cC'  \subset \lowerIdeal{\cC}$ and 
\[
n+1-k_{\cC'} \le n+1 - k_{\tilde{\cC} } =  \sum_{T \in  \tilde{\cC}}  i_T  =  \sum_{T \in  \cC' } i_T \le  \sum_{T\in \langle \cC'\rangle \cap \cC}j_T.   
\]
It follows that $B$ is not in $M\cG_{\cC} (n)$. This is a contradiction and concludes the argument. 
\end{proof}

We now simplify the description of $M\cG_{\cC}(n)$ by showing that it suffices to test the  condition in \Cref{def: minGenSet 1 orbit} for fewer antichains $\cC'$. 

\begin{prop}\label{prop: simplified MG_C}
If $\cC\subset 2^{[c]}$ is an antichain with $k_\cC\geq 1$ then one has 
\begin{align}\label{eqn: minGenSet 1 orbit}
M\cG_{\cC}(n) = & \left \{ \newMatrix{\cC}{j} \;  \mid \;  j_T \geq 1, \sum_{T\in\cC} j_T = n+1-k_{\cC} \ \text{ and }   \right. \\
& \hspace*{2cm} \left. \sum_{T\in \cC- \cD} j_T > k_{\cD}-k_{\cC} \text{ for any subset } \varnothing \neq \cD \subsetneqq \cC \right\}.   \nonumber
\end{align}
\end{prop}

\begin{proof} 
Fix a subset $\cD \neq \varnothing$  with $\cD \neq \cC$ and consider the set $\cU$ of antichains  $\cC'\subseteq \lowerIdeal{\cC} - \{\varnothing, \cC\}$ with $\cD = \cC \cap \langle \cC' \rangle$. Note that $\cD$ is in $\cU$. Furthermore, for any $\cC' \in \cU$, one has $\cC - \ideal{\cC'} = \cC - \cD$ and $\cD \subseteq \ideal{\cC'}$.  The latter containment gives $\ideal{\cD} \subseteq \ideal{\cC'}$, which implies 
$k_{\cD} \ge k_{\cC'}$. Hence, for $\cC' \in \cU$, the condition $\sum_{T\in \cC- \ideal{\cC'}}j_T > k_{\cC'}-k_{\cC}$ is strongest if $\cC' = \cD$.

It remains to show that $\Supp(\newMatrix{\cC}{j})=\cC$ for all $\newMatrix{\cC}{j}\in M\cG_{\cC}(n)$. If $j_T=0$ for some $T\in \cC$, then, setting $\cD=\cC-\{T\}$, we have
$\sum_{S\in\cC-\cD}j_S=0$ on the one side and  $k_{\cD}-k_{\cC}\geq 0$ on the other side, which is a contradiction to the second condition.
\end{proof}

We illustrate the above results by considering again \Cref{ex: one-orbit}. 

%

\begin{example}
Recall the matrix $A$ from Example \ref{ex: one-orbit}.
One can now completely describe the set of minimal generators of $I_n^\vee=\langle \Sym(n)\cdot \xx^A\rangle^\vee$ as follows: The minimal generators fall into four classes. For simplicity, we only specify those antichains $\cC$ that contribute minimal generators and the relevant conditions for specifying matrices in $M\cG_{\cC} (n)$. 
\begin{itemize}
    \item[(1)] The antichain $\cC = \{12,13,23\}$ gives the equation $j_{12}+j_{13}+j_{23}=n-2$. There are $\binom{n}{2}$ orbits of  such minimal generators.
    \item[(2)] The antichains of the form $\cC = \{i\}$ for $i\in [3]$ correspond to the equation $j_{i}=n-1$.  There are three orbits of  such minimal generators.
    \item[(3)] The antichains of the form $\cC = \{i,123- i\}$ with $i \in [3]$ correspond to the equation $j_i+j_{\{123\}- \{i\}}=n-1$ and the inequalities $j_i\geq 2$, $j_{\{123\}- \{i\}}\geq 1$. There are $3(n-3)$ orbits of  such minimal generators.
    \item[(4)] The antichains of the form $\cC = \{i,k\}$ with $1\leq i< k\leq 3$ correspond to the equation $j_i+j_k=n$ and the inequalities $j_i\geq 2$ and $j_k\geq 2$. There are $3(n-3)$ orbits of  such minimal generators.
\end{itemize}
In total, the number of orbits of minimal generators of $I_n^\vee$ equals $\frac{n^2}{2}+\frac{11}{2}n-15$.

\end{example}

\section{The General Case}\label{sec: generalCase}

Given any $\Sym$-invariant chain $(I_n)_{n \in \NN}$ of squarefree monomial ideals, we produce a generating set of the Alexander dual $I_n^{\vee}$ for any $n \gg 0$. We proceed in several steps. Using the results of the previous section, we first obtain a rather redundant generating set. In a second step, we trim down this set by identifying additional properties of any minimal generator of $I_n^{\vee}$. Since the resulting generating set is still not necessarily minimal, we need to adjust it to achieve minimality.

\begin{notation} 
    \label{nota:set-up general case}
Throughout this section, fix positive integers $c$ and $m$ as well as $0 - 1$-matrices $A_1,\ldots,A_s$ of the form 
$$
A_i = \left[\Ast_{\varnothing \neq T\in 2^{[c]}} A_T (k_T^{(i)})\right] 
$$
with $c$ rows and at most $m$ columns, where $k_T^{(i)}\in \mathbb{N}_0$. Furthermore, let $(I_n)_{n \in \NN}$ be the $\Sym$-invariant chain of ideals generated by $\xx^{A_1},\dots, \xx^{A_s}$, that is, $I_n = 0$ if $n < m$ and 
$I_n = \langle \Sym(n)\cdot(\xx^{A_1},\dots, \xx^{A_s}) \rangle \subset R_n = \mathbf{k}[x_{i,j} \mid i\in [c], j\in [n] ]$ if $n \ge m$. 
When thinking of $A_i$ as an exponent matrix of a monomial in $R_n$, we abuse notation and follow our convention by considering $A_i$ as a $c \times n$ matrix padded with zero columns, that is, we use 
$$
A_i = \left[\Ast_{T\in 2^{[c]}} A_T (k_T^{(i)})\right] \in \{0,1\}^{c\times n}
$$
with $k_{\varnothing}^{(i)} = n - \sum_{\varnothing \neq T\in 2^{[c]}} k_T^{(i)}$. 
\end{notation}

We now outline a construction that produces a generating set of $I_n^{\vee}$. 

\begin{construction}\label{const:generalCase} 
Let $n \ge m$ and consider a matrix
$$
B = \left[\Ast_{T\in 2^{[c]}} A_T(\ell_T)\right] \in \{0,1\}^{c\times n}
$$ 
with $\ell_T\in\mathbb{N}_0$. Using the fact that $I_n^\vee = \bigcap_{i\in [s]} \left\langle \Sym(n)\cdot \xx^{A_i}\right\rangle^\vee$ and the Alexander Dual Membership Criterion, it follows that 
$\displaystyle{\xx^B}$ is in  $I_n^\vee$  if and only if for every $i\in [s]$, there is an order ideal  $\varnothing  \neq J_i \subsetneqq 2^{[c]}$ with
$$
    \sum_{T\in J_i} k_T^{(i)} > \sum_{T\in J_i} \ell_{T^C}.
$$
Hence, to obtain a generating set of $I_n^\vee$, fix any $s$-tuple $\jJ = (J_1,\dots, J_s)$ of proper order ideals of $2^{[c]}$, and observe that it determines a set of monomials $\xx^B\in I_n^\vee$ satisfying the system of inequalities
\begin{equation}\label{eq: generalIneqs}
    \sum_{T\in J_i} k_T^{(i)} > \sum_{T\in J_i} \ell_{T^C} = \sum_{S\in \overline{J}_i} \ell_S,  \quad i\in [s].
\end{equation}

Define $\cR \coloneqq \overline{J}_1\cup\dots \cup \overline{J}_s$. 
Now, fix a solution $L = (\ell_S \; \mid \; S\in \cR) \in \NN_0^{|\cR|}$ of the system of inequalities (\ref{eq: generalIneqs}). In order to describe a matrix $B$ with $\xx^B \in I_n^\vee$, it remains to determine the columns of $B$ whose support is not in $\cR$. Let $\cC_\jJ$ be the antichain generating $2^{[c]}- \cR$ which is an order ideal by \Cref{lem: complement of order ideal}. By a similar argument as in the proof of part (a) of the Exchange \Cref{lem: exchangeLemma}, it follows that for the description of a minimal generating set of $I_n^\vee$ it suffices to consider matrices $B$ whose columns (with support not in $\cR$) are supported on minimal elements of $2^{[c]}- \cR$, i.e.,  $\ell_T = 0$ if $T\notin \cC_\jJ$.
\end{construction}

The following result is now immediate.

\begin{prop}\label{prop: generalGenSet} Adopt the notation and hypotheses of Construction \ref{const:generalCase}.
For $n \ge m$, the ideal $I_n^\vee$ is generated by the $\Sym(n)$-orbits of monomials $\xx^B$ with $B$ in the set
\begin{equation}\label{eq: generalCaseRedundant}
\bigcup_{\substack{\jJ = (J_1,\dots, J_s) \\ \varnothing \neq J_i \subsetneqq 2^{[c]}}} \left\{ \left[ \Ast_{S\in \cR} A_S(\ell_S)\right] \ast \left[\Ast_{S\in \cC_\jJ} A_S(\ell_S)\right] \; \;  \mid  \;\; \substack{\ell_S \ge 0, \ L = (\ell_S \; \mid \;  S\in \cR) \text{ a solution of } (\ref{eq: generalIneqs}),\\ \sum_{S\in \cC_\jJ} \ell_S = n-\sum_{S\in \cR} \ell_S} \right\}.
\end{equation}
\end{prop}

The above generating set is typically not minimal. The following example gives a hint on how to trim down the generating set. 

\begin{example}\label{ex: two-Orbit} Fix matrices
$$
A_1 = \begin{bmatrix} 1 & 1& 1 \\ 1 & 1& 0 \\ 0 & 0 & 1
\end{bmatrix}, \qquad 
A_2 = \begin{bmatrix} 1 & 0 & 0 \\ 1 & 1 & 1 \\ 0 & 1 & 1
\end{bmatrix}
$$
and consider the ideal $I_n = \langle \Sym(n)\cdot \xx^{A_1}, \; \Sym(n)\cdot \xx^{A_2}\rangle$. For order ideals $J_1 = \langle 12,13\rangle $ and $J_2 = \langle 12, 23\rangle$ form the tuple $\jJ = (J_1, J_2)$. Here, we omit commas and parentheses around and between elements of a set. In this case, $\cR = \overline{J}_1 \cup \overline{J}_2 = \{\varnothing, 1,2,3\}$ and $2^{[c]} - \cR$ is generated by the antichain $\cC_\jJ = \{12, 13, 23\}$.
The system of equations corresponding to Equation (\ref{eq: generalIneqs}) for this choice of $\jJ$ is
\begin{align*}
    k_{12}^{(1)} + k_{13}^{(1)} + k_{123}^{(1)} = 3 &> \ell_3 + \ell_2 + \ell_\varnothing \\
    k_{12}^{(2)} + k_{23}^{(2)} + k_{123}^{(2)} = 3 &> \ell_3 + \ell_1 + \ell_\varnothing.
\end{align*}
In order to avoid exponent matrices from (\ref{eq: generalCaseRedundant}) arising from solutions of the above inequalities that give non-minimal generators, we maximize the number of zero columns by setting $\ell_\varnothing$ equal to $2$. This forces $\ell_1 = \ell_2 = \ell_3 = 0$ and
$$
\sum_{S\in \cC_\jJ} \ell_S = \ell_{12} + \ell_{13} + \ell_{23} = n - \sum_{S\in \cR} \ell_S = n - 2.
$$
For instance, in the case where $n = 4$, this equation gives the following $6$ orbit generators of exponent matrices for monomials in $I_4^{\vee}$:
\begin{align*}
& \begin{bmatrix}
0 & 0 & 1 & 1 \\ 0 & 0 & 1 & 1 \\ 0 & 0 & 0 & 0
\end{bmatrix}, \quad
\begin{bmatrix}
0 & 0 & 1 & 0 \\ 0 & 0 & 1 & 1 \\ 0 & 0 & 0 & 1
\end{bmatrix}, \quad
\begin{bmatrix}
0 & 0 & 1 & 1 \\ 0 & 0 & 1 & 0 \\ 0 & 0 & 0 & 1
\end{bmatrix}, \quad \\
& \begin{bmatrix}
0 & 0 & 1 & 0 \\ 0 & 0 & 0 & 1 \\ 0 & 0 & 1 & 1
\end{bmatrix}, \quad
\begin{bmatrix}
0 & 0 & 1 & 1 \\ 0 & 0 & 0 & 0 \\ 0 & 0 & 1 & 1
\end{bmatrix}, \quad
\begin{bmatrix}
0 & 0 & 0 & 0 \\ 0 & 0 & 1 & 1 \\ 0 & 0 & 1 & 1
\end{bmatrix}.
\end{align*}
\end{example}


We need some notation in order to identify non-minimal generators of $I_n^{\vee}$. 

\begin{notation}\label{not: restrictionLemma} For any $s$-tuple $\jJ = (J_1,\dots, J_s)$ of proper order ideals of $2^{[c]}$ and any $\Lambda\subseteq [s]$, define $D_{\jJ,\Lambda}$ to be the set of minimal elements of 
$$\bigcap_{k\in \Lambda} \overline{J}_k - \bigcup_{k\in [s]-\Lambda} \overline{J}_k$$ if $\Lambda\neq \varnothing$, and the  set of minimal elements of
$$
2^{[c]} - \overline{J}_1\cup\dots\cup\overline{J}_s = 2^{[c]} - \cR
$$
if $\Lambda = \varnothing$.
\end{notation}

\begin{remark}\label{rem: DspecialCases}
Note that $D_{\jJ,[s]} = \{\varnothing\}$ and $D_{\jJ,\varnothing}$ is the antichain generating the order ideal $2^{[c]} -\cR$.
\end{remark}

The following result is a first pass at a description of the minimal generators of $I_n^\vee$.

\begin{lemma}[Restriction Lemma]\label{lem: restrictionLemma} 
If $\xx^B$ is a minimal generator of $I_n^\vee$, then $B$ is of the form
$$
B = \left[ \Ast_{S\in F} A_S(\ell_S)\right]\ast\left[\Ast_{S\in \cC} A_S(\ell_S)\right] \in \{0,1\}^{c\times n}, 
$$
where:
\begin{enumerate}[(i)]
   
    \item $\cC = D_{\jJ,\varnothing}$ for some $s$-tuple of proper order ideals $\jJ$;
    \item $\displaystyle{F\subseteq \sqcup_{\varnothing\neq\Lambda\subseteq[s]} D_{\jJ,\Lambda}}$;
    \item $F\cap \langle \cC\rangle = \varnothing$; 
    \item $F\subset \Supp B$, i.e., $\ell_S > 0$ if $S \in F$; and
     \item setting $\ell_S = 0$ for any $S\in \cR - F$, the $|\cR|$-tuple  $(\ell_S : S\in \cR) \in \NN_0^{|\cR|}$ is a solution of the system of inequalities (\ref{eq: generalIneqs}). 
\end{enumerate}
\end{lemma}

\begin{proof} 
Proposition \ref{prop: generalGenSet} implies Properties $(i)$ and $(v)$ as well as $\Supp B \subseteq \cR \cup \cC$.  
Next we establish Property $(ii)$. Note that
$$
\cR = \overline{J}_1\cup\dots \cup \overline{J}_s = \bigsqcup_{\varnothing\neq\Lambda \subseteq [s]} \left(\bigcap_{k\in\Lambda} \overline{J}_k - \bigcup_{k\in [s] - \Lambda} \overline{J}_k\right).
$$
By a similar argument as in the proof of part (a) of the Exchange Lemma \ref{lem: exchangeLemma}, we see that $\Supp B \subseteq F \sqcup \cC$ for some 
$F\subseteq \sqcup_{\varnothing\neq \Lambda\subseteq [s]} D_{\jJ,\Lambda}$. As making $F$ smaller does not affect (ii) and (iii), we may assume that $F\subset \Supp B$.

 Property (iii) follows from the following chain of containments
$$
F \subseteq \sqcup_{\varnothing\neq\Lambda\subseteq[s]} D_{\jJ,\Lambda} \subseteq \overline{J}_1\cup\cdots\cup\overline{J}_s=\cR=2^{[c]}-\langle \cC\rangle.
$$
%
\end{proof}

\begin{notation}\label{not: almostMinimalGenSet} Recall the definition of $D_{\jJ, \Lambda}$ from Notation \ref{not: restrictionLemma}. Denote by $\cP$ the set of pairs $(\cF, \cC)$ such that there is an $s$-tuple $\jJ = (J_1,\dots, J_s)$ of proper order ideals of $2^{[c]}$, where $\cC = D_{\jJ,\varnothing}$ is an antichain and $\cF = (\ell_S \in \NN \; \mid \;  S\in F)$ for some subset $F \neq \varnothing$ of $2^{[c]}$ satisfying the following conditions:
\begin{itemize}
    \item $F\cap \langle \cC\rangle = \varnothing$,
    \item $F\subseteq \bigsqcup_{\varnothing \neq \Lambda \subseteq [s]} D_{\jJ,\Lambda}$, and
    \item $\cF$ can be extended to a solution $L = (\ell_S : S\in \overline{J}_1\cup\dots \cup \overline{J}_s)$ of the system of Inequalities \eqref{eq: generalIneqs} by setting $\ell_S = 0$ if $S \in \overline{J}_1\cup\dots \cup \overline{J}_s - F$. 
\end{itemize}
\end{notation}

\begin{definition}\label{def: generalGenSet}
For any integer $n \ge m$ and any pair $(\cF,\cC)\in \cP$, define the set
\begin{align*}
\cG_{\cF,\cC}(n)&= \left\{ \left[ \Ast_{S\in F} A_S(\ell_S) \right] \ast \left[ \Ast_{S\in\cC} A_S(\ell_S)\right] \; \mid \; \sum_{S\in\cC}\ell_S = n - \sum_{S\in F}\ell_S \text{ and } \ell_S\geq 0 \text{ if } S\in \cC\right\} \\
&=\left\{ M_{\cF} \ast \left[ \Ast_{S\in\cC} A_S(\ell_S)\right] \; \mid \; \sum_{S\in\cC}\ell_S = n - \sum_{S\in F}\ell_S \text{ and } \ell_S\geq 0 \text{ if } S\in \cC\right\},
\end{align*}
where $M_{\cF}\coloneqq \left[ \Ast_{S\in F} A_S(\ell_S) \right] $.
\end{definition}
 
 Some observations are in order. 
 
\begin{remark}\label{obs: usefulFactsAboutGs} \mbox{}
\begin{enumerate}[(i)] 
    \item By definition, the support of the matrix $M_{\cF}$ is $F$. Sometimes we will refer to $F$ also as the support of $\cF$. 
    \item Any matrix $B\in \cG_{\cF,\cC} (n)$ satisfies $F \subset \Supp B \subseteq F\sqcup \cC$. 
    
    \item The set $\cP$ is finite as, for fixed $c$ and $s$, there are only finite $s$-tuples of order ideals and, for each $s$-tuple $\jJ$, the system of inequalities (\ref{eq: generalIneqs}) has only finitely many solutions with entries in $\NN_{0}$.
\end{enumerate}
\end{remark}

\begin{cor} 
     \label{cor:smaller gen set}
For $n \ge m$, the ideal $I_n^\vee$ is generated by the $\Sym(n)$-orbits of monomials $\xx^B$ with $B$ in the set
 $\bigcup_{(\cF,\cC)\in \cP} \cG_{\cF,\cC} (n)$.
\end{cor} 

\begin{proof} 
The given set is obtained from the generating set in \Cref{prop: generalGenSet} by applying the Restriction Lemma \ref{lem: restrictionLemma}. 
\end{proof} 

We now describe, for any sufficiently large $n$, a minimal generating set of the Alexander dual of the ideal $I_n$. In Section \ref{sec: counting},   we will study the cardinalities of these generating sets as a function in $n$. 

\begin{theorem}
    \label{thm:min gens in general} 
If $n \ge m$ then the Alexander dual $I_n^{\vee}$ of $I_n$ is minimally generated by the $\Sym (n)$-orbits of monomials $\xx^B$ with $B$ in the set 
\begin{equation}\label{eq:minGens}
    \cG(I_n^\vee)=\bigcup_{(\cF,\cC)\in \cP} \big [ \cG_{\cF,\cC}(n) - \bigcup_{\substack{(\cF',\cC'))\in \cP\\ (\cF,\cC)\neq (\cF',\cC')}} \left(I_{\cF',\cC'}(n)- \cG_{\cF',\cC'}(n)\right) \big ], 
\end{equation}
where 
\begin{equation*}
   I_{\cF,\cC}(n)=\{B \in \{0, 1\}^{c \times n} \; \mid \;  \xx^B\in  \langle \xx^A \; \mid \; A\in \cG_{\cF,\cC}(n) \rangle \subset R_n\}.
\end{equation*}   
\end{theorem}

\begin{proof}
The Divisibility \Cref{lem: divisibility up to symmetry} implies that any set $\cG_{\cF,\cC}(n)$ does not contain two distinct matrices $B, B'$ such that $\xx^B$ divides $\xx^{B'}$ up to symmetry. It follows that  $\cG(I_n^\vee)$ does not contain redundant matrices. It remains to show that the orbits of $\xx^B$ with $B \in \cG(I_n^\vee)$ generate $I_n^{\vee}$. 

By \Cref{cor:smaller gen set}, we know that the orbits of $\xx^B$ with $B \in\bigcup_{(\cF,\cC)\in \cP} \cG_{\cF,\cC} (n)$ generate $I_n^{\vee}$. Hence, it suffices to show, for any $B \in \cG_{\cF,\cC}(n)\cap \bigcup_{\substack{(\cF',\cC'))\in \cP\\ (\cF,\cC)\neq (\cF',\cC')}}   \left(I_{\cF',\cC'}(n)- \cG_{\cF',\cC'}(n)\right)$, there is some matrix 
$B_1 \in \cG(I_n^\vee)$ such that $\xx^{B_1}$ properly divides $\xx^B$ up to symmetry, and so in particular $B_1 \neq B$. Indeed, by choice of $B$, there is some $(\cF_1, \cC_1) \neq (\cF, \cC)$ in  $\cP$ such that $B_1 \in \cG_{\cF_1,\cC_1}(n)$ and  $\xx^{B_1}$ properly divides $\xx^B$ up to symmetry. 
If $B_1$ is not contained in any set $I_{\cF_2,\cC_2}(n)- \cG_{\cF_2,\cC_2}(n)$ with $(\cF_2, \cC_2) \neq (\cF_1, \cC_1)$, we get $B_1 \in  \cG(I_n^\vee)$ and are done. Otherwise, there is a matrix $B_2$ in  $ \cG_{\cF_2,\cC_2}(n)$ such  $\xx^{B_2}$ properly divides $\xx^{B_1}$ up to symmetry. Recall that $\xx^{B_1}$ properly divides $\xx^{B}$ up to symmetry. Since any such chain of proper divisors of $\xx^B$ is finite, eventually we get a matrix $B' \in \cG(I_n^\vee)$ such that $\xx^{B'}$ properly divides $\xx^{B}$ up to symmetry, which concludes the argument. 
\end{proof}

\begin{example}
      \label{ex: two-Orbit yet again} 
Continuing Example \ref{ex: two-Orbit}, let $c = 3$ and consider,  for any $n \ge 3$, the ideal $I_n = \langle \Sym(n)\cdot \xx^{A_1}, \; \Sym(n)\cdot \xx^{A_2}\rangle$, where $$
A_1 = \begin{bmatrix} 1 & 1& 1 \\ 1 & 1& 0 \\ 0 & 0 & 1
\end{bmatrix}, \qquad 
A_2 = \begin{bmatrix} 1 & 0 & 0 \\ 1 & 1 & 1 \\ 0 & 1 & 1
\end{bmatrix}. 
$$
In $2^{[3]}$, there are 20 order ideals. It turns out that only five pairs of order ideals lead to minimal generators. 
We list them and their corresponding conditions on the matrices $[\Ast_{S \in F \sqcup \cC} A_S (\ell_S)] \in \cG_{\cF, \cC}$ mentioning only the numbers $\ell_S$ that can be not zero. 
\begin{center}
\begin{tabular}{|c|c|c|} \hline
$\jJ = (J_1, J_2)$ & Conditions & \# minimal generators \\ \hline \hline
$(\langle 12\rangle,\langle23\rangle)$ & $\ell_\varnothing = 1, \; \ell_2 = n-1 \ge 3,$ & 1 \\ \hline
$( \langle12\rangle,\langle12\rangle)$ & $\ell_1 \geq 3,  \; \ell_2 \geq 0, $ & $n-2$ \\
& $\ell_1 + \ell_2  = n$ &\\ \hline
$(\langle12\rangle,\langle12,23\rangle)$ & $\ell_\varnothing = \ell_1 = 1,  \; \ell_2 \geq 1,  \; \ell_{13}\geq 0, $ & $n-2$ \\
& $ \ell_2 + \ell_{1 3} = n-2 $ & \\ \hline
$(\langle13\rangle,\langle23\rangle)$ & $\ell_3 = n  \geq 3,$ & 1 \\ \hline
$(\langle13\rangle,\langle23\rangle$ & $\ell_3 \geq 2,  \; \ell_1 = 1,  \; \ell_{12} \geq 0$ & $n-2$ \\
& $ \ell_3 + \ell_{1 2} = n-1 $ & \\ \hline
$(\langle12,13\rangle,\langle12,23\rangle)$ & $\ell_\varnothing = 2,  \; \ell_{12} \geq 0,  \; \ell_{13} \geq 0,  \; \ell_{23}\geq 0$ & $\binom{n}{2} = \frac{n(n-1)}{2}$ \\
& $ \ell_{1 2} + \ell_{1 3} + \ell_{2 3}  = n-2 $ & \\ \hline
\end{tabular}
\end{center}
We note that the pair of order ideals $(\langle13\rangle,\langle23\rangle)$ gives two different classes of minimal generators coming from different solutions of the system of inequalities \eqref{eq: generalIneqs}.
In total, this gives that the number of minimal orbit generators of the Alexander dual is $\frac{n^2}{2} + \frac{5}{2}n - 4$.
\end{example}

The polynomial growth of the number of orbit generators is true in general as we show in Section \ref{sec: counting}.  The argument  uses a result from convex geometry that we establish in the following section.  

\section{Cone Decompositions} \label{sec: coneDecomp}


We consider a particular class of polyhedra. Their characteristic feature is that the bounding hyperplanes are defined by conditions on sums of variables. Our main result is that any such polyhedron can be decomposed as union of pointed rational cones. Thus, Ehrhart theory can be applied to study the number of integer points of the polyhedron. 


\begin{notation}\label{not: polyhedron}
Given any $k\in \NN$,  fix $\aa=(a_T ~\mid~  T \subseteq [k]) \in \ZZ^{2^{[k]}-\{\varnothing\}}$ and $\bb=(b_T ~\mid~: T\in U)\in \ZZ^{2^{|U|}}$ for some fixed $U\subseteq 2^{[k]}$. Define a polyhedron $P_{\aa,\bb}$ as the set 
\begin{equation}\label{eq: polyhedron}
P_{\aa,\bb} := \left\{(x_1,\ldots,x_k) \in \RR^k ~\mid~ \sum_{i\in T}x_i \geq a_T \text{ for all } \varnothing \neq T \subseteq [k], \sum_{i\in T}x_i\leq b_T\text{ for all } T\in U\right\}.
\end{equation}
If $U=\emptyset$, we write $P_{\aa}$ for short.
\end{notation}

The goal of this section is to determine the cardinality of the set of integer points of $P_{\aa,\bb}$.
\begin{theorem}\label{prop: disjoint cone decomposition}
Adopt Notation \ref{not: polyhedron}.
Either $P_{\aa,\bb}$ is empty or there exist disjoint, pointed rational cones $C_1,\ldots,C_t\subset \RR^k$ with integral apices such that
$$P_{\aa,\bb} \cap \ZZ^k = \bigcup_{i=1}^t (C_i \cap \ZZ^k).$$
Moreover, each cone $C_i$ is of the form 
$$C_i = \left \{(x_1,\ldots,x_k)\in \RR^k~\mid~ x_i = c_i \text{ for }i\in J, \; x_i \geq c_i \text{ for }i\in[k]- J\right\}$$
for some $J\subseteq[k]$ and $(c_1,\ldots,c_k)\in\NN^k$. We will call such a cone an \emph{orthant}.
\end{theorem}

Before proving this theorem, we demonstrate it with an example in the simplest case when $U = \varnothing$.

\begin{example}\label{ex: coneDecomp}
Let $k=3$ and $U = \varnothing$. Fix the vector $\aa\in \ZZ^{2^{[3]}}$ such that $a_{\{1\}} =  a_{\{2\}}= a_{\{3\}} = 1$,  $a_{\{1,2\}} = 3$, and $a_T = 0$ for all other $T\subseteq [3]$. Figure~\ref{fig:one_orbit_polyhedron} shows both $P_\aa$ and its cone decomposition.
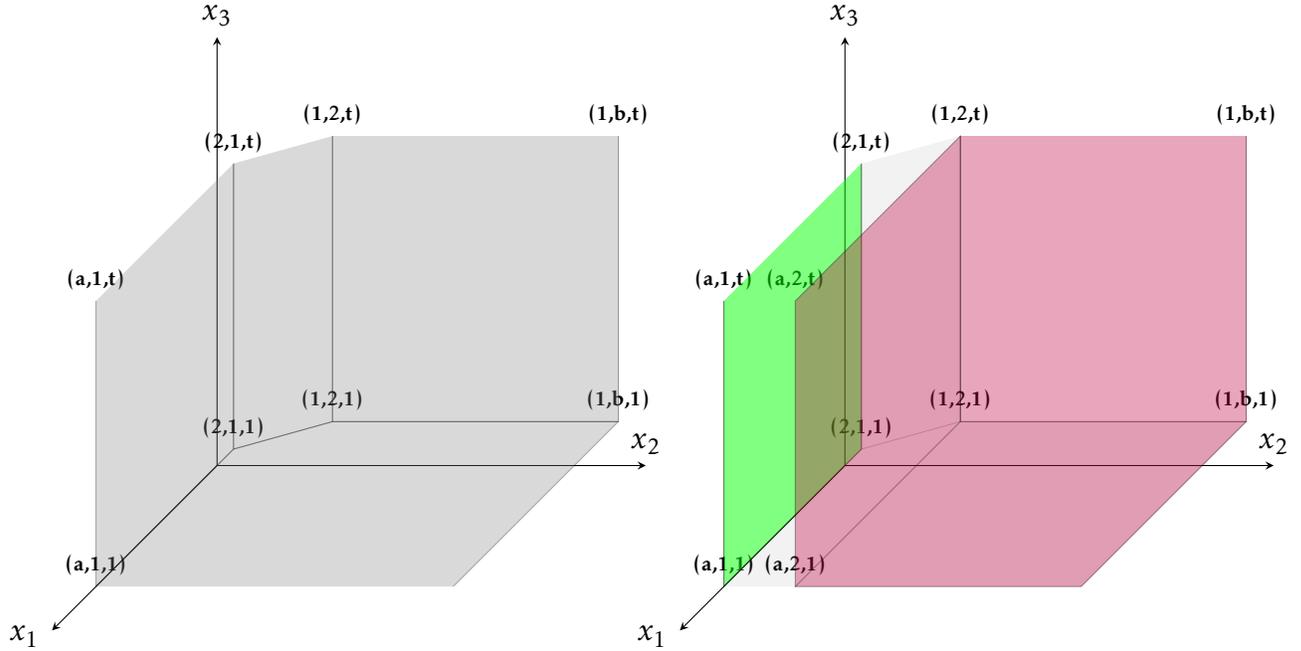
\begin{figure}[h]
    \centering

 \begin{tikzpicture}
 
 \begin{scope}[scale=0.95, xshift=0, yshift=0]
\draw[thin, -stealth] (0,0,0)--(0,0,6); 
\draw[thin, -stealth] (0,0,0)--(0,6,0);
\draw[thin, -stealth] (0,0,0)--(6,0,0);

\coordinate[label=$x_1$] (x1) at (0,0,7);
\coordinate[label=$x_2$] (x2) at (6,0,0);
\coordinate[label=$x_3$] (x3) at (0,6,0);

 \coordinate[label={\bf\tiny(2,1,1)}] (A) at (1,1,2);
 \coordinate[label={\bf\tiny(1,2,1)}] (B) at (2,1,1);
 \coordinate[label={\bf\tiny(2,1,t)}] (C) at (1,5,2);
 \coordinate[label={\bf\tiny(1,2,t)}] (D) at (2,5,1); 
 \coordinate[label={\bf\tiny(a,1,1)}] (E) at (1,1,7); 
 \coordinate[label={\bf\tiny(a,1,t)}] (F) at (1,5,7);
 \coordinate[label={\bf\tiny(1,b,1)}] (G) at (6,1,1); 
 \coordinate[label={\bf\tiny(1,b,t)}] (H) at (6,5,1);
 \coordinate[] (I) at (6,1,7); 

 \draw[fill=gray,opacity=0.3] (C)--(A)--(B)--(D);
 \draw[fill=gray,opacity=0.3] (F)--(E)--(A)--(C);
 \draw[fill=gray,opacity=0.3] (D)--(B)--(G)--(H);
 \draw[fill=gray,opacity=0.3] (E)--(A)--(B)--(G)--(I);
\end{scope}

\begin{scope}[scale=0.95, xshift=250, yshift=0]
\draw[thin, -stealth] (0,0,0)--(0,0,6); 
\draw[thin, -stealth] (0,0,0)--(0,6,0);
\draw[thin, -stealth] (0,0,0)--(6,0,0);

\coordinate[label=$x_1$] (x1) at (0,0,7);
\coordinate[label=$x_2$] (x2) at (6,0,0);
\coordinate[label=$x_3$] (x3) at (0,6,0);

 \coordinate[label={\bf\tiny(2,1,1)}] (A) at (1,1,2);
 \coordinate[label={\bf\tiny(1,2,1)}] (B) at (2,1,1);
 \coordinate[label={\bf\tiny(2,1,t)}] (C) at (1,5,2);
 \coordinate[label={\bf\tiny(1,2,t)}] (D) at (2,5,1); 
 \coordinate[label={\bf\tiny(a,1,1)}] (E) at (1,1,7); 
 \coordinate[label={\bf\tiny(a,1,t)}] (F) at (1,5,7);
 \coordinate[label={\bf\tiny(1,b,1)}] (G) at (6,1,1); 
 \coordinate[label={\bf\tiny(1,b,t)}] (H) at (6,5,1);
 \coordinate[] (I) at (6,1,7); 
 \coordinate[label={\bf\tiny(a,2,1)}] (J) at (2,1,7);
 \coordinate[label={\bf\tiny(a,2,t)}] (K) at (2,5,7); 

 \draw[fill=gray,opacity=0.1] (C)--(A)--(B)--(D);
 \draw[fill=green,opacity=0.5] (F)--(E)--(A)--(C);
 \draw[fill=purple,opacity=0.35] (D)--(B)--(G)--(H);
 \draw[fill=gray,opacity=0.1] (E)--(A)--(B)--(G)--(I);
 \draw[fill=purple, opacity=0.35] (B)--(D)--(K)--(J);
 \draw[fill=purple, opacity=0.35](B)--(J)--(I)--(G);
 \end{scope}
\end{tikzpicture}
    \caption{
    The polyhedron $P_\aa$ from Example \ref{ex: coneDecomp} and its decomposition into orthants.
    }
    \label{fig:one_orbit_polyhedron}
\end{figure}
\end{example}

\begin{proof}[Proof of \Cref{prop: disjoint cone decomposition}]
We assume that $a_{\varnothing}\leq 0$, since otherwise $P_{\aa,\bb}=\emptyset$. In the former case however, does not impose any further restrictions on $P_{\aa,\bb}$ which is why we will not mention it further in our considerations and just assume that $\aa\in \ZZ^{2^{[k]}-\{\varnothing\}}$.

First consider the case that $U=\varnothing$. 
We prove the statement by induction on $k$. When $k=1$, $P_\aa$ has the form 
$$P_\aa = \{x\in\RR~ \mid~ x\geq a\}$$
for some $a\in \ZZ$. This is itself a cone, so the statement holds. 
Now fix $k\geq 2$ and $\aa\in\ZZ^{2^{[k]}-\{\varnothing\}}$. We decompose the integer points in $P_\aa$ by their $x_k$ coordinate. First, set $\widetilde{a}_{\{k\}}:=\max\{a_T ~\mid~ T\subseteq [k], \; k\in T\}$. Define $\widetilde{P}_\aa$ by
$$\widetilde{P}_\aa\coloneqq\left\{(x_1,\ldots,x_k)\in\RR^k~\mid~ x_k \geq \widetilde{a}_{\{k\}}, \sum_{i\in T} x_i \geq a_T \text{ for all } \varnothing \neq T \subseteq [k]\right\}.$$
In other words, $\widetilde{P}_\aa$ is the part of $P_\aa$ lying on or above the plane  $\{x_k = \widetilde{a}_{\{k\}}\}$. Now we consider the points in $P_\aa$ that do not lie in $\widetilde{P}_\aa$. By construction, any such point $(x_1,\ldots,x_k)$ must satisfy $a_{\{k\}}\leq x_k <\widetilde{a}_{\{k\}}$. If such a point is also an integer point, it must lie on some plane $\{x_k = i\}$, for some integer $i$ with $a_{\{k\}}\leq i \leq \widetilde{a}_{\{k\}}-1$. Hence, we define for each such $i$ the polyhedron $P_\aa^{(i)}$ as follows:
$$P_\aa^{(i)} \coloneqq \left\{ (x_1,\ldots,x_k)\in\RR^k~\mid~ x_k = i, \sum_{j\in T}x_j \geq \max(a_T,a_{T\cup\{i\}}-i)\text{ for all }\varnothing \neq T \subseteq [k-1]\right\}.$$
Note that $P_\aa^{(i)}$ is simply the intersection of $P_\aa$ with the plane $\{x_k=i\}$, since the inequalities from $P_\aa$ of the form $x_k + \sum_{j\in T}x_j \geq a_{T\cup\{k\}}$ become $\sum_{j\in T}x_j \geq a_{T\cup\{k\}}-i$ for each $\varnothing \neq T \subseteq [k-1]$. It follows that 
$$P_\aa\cap \ZZ^k = (\widetilde{P}_\aa \cap \ZZ^k) \cup \left( \bigcup_{i=a_{\{k\}}}^{\widetilde{a}_{\{k\}-1}} (P_\aa^{(i)}\cap \ZZ^k)\right)$$
and that this union is disjoint. Moreover, we can write
$$\widetilde{P}_\aa = \left\{x\in\ZZ^{k-1}~\mid~\sum_{i\in T}x_i \geq a_T\text{ for all }\varnothing \neq T \subseteq [k-1]\right\} \times [\widetilde{a}_{\{k\}},\infty)$$
and 
$$P_\aa^{(i)} = \left\{x\in \ZZ^{k-1}~\mid~ \sum_{j\in T}x_j \geq \max\{a_T, a_{T\cup\{k\}}-i\} \text{ for all } \varnothing \neq T \subseteq[k-1]\right\} \times\{i\}.$$
 Call the first factor in each Cartesian product $\widetilde{P}_{\aa,1}$ and $P_{\aa,1}^{(i)}$, respectively. By induction, we may decompose the integer points of $\widetilde{P}_{\aa,1}$ and $P_{\aa,1}^{(i)}$ into a disjoint union of integer points of pointed rational orthants. The claim follows.
 
 Now assume that $U\neq \varnothing$. Let $U=\{T_1,\ldots,T_r\}$. The points in $P_{\aa,\bb}$ satisfy 
\begin{equation}\label{eq: system of equations}
    a_{T_i}\leq \sum_{j\in T_i}x_j\leq b_{T_i} \qquad \text{for all } 1\leq i\leq r.
\end{equation}
If this system of inequalities as a system of inequalities in $\RR^{T_1\cup\cdots\cup T_r}$ does not have integral solutions, then $P_{\aa,\bb}=\varnothing$. Otherwise, we are interested in  integral solutions of \eqref{eq: system of equations} with $x_j\geq a_j$ for all $j\in T_1\cup\cdots \cup T_r$. Since componentwise those must be bounded from above as well, it follows that there are only finitely many such (or none). Denote the set of those solutions by $\cL$ and assume $\cL\neq \varnothing$ since otherwise $P_{\aa,\bb}=\varnothing$. Given $w\in \cL$ we set
$$
P_{\aa,\bb}(w)\coloneqq=\{x\in P_{\aa,\bb}~\mid~x_j=w_j\text{ for } j\in T_1\cup\cdots\cup T_r\}.
$$
We then have
$$
P_{\aa,\bb}\cap \ZZ^k=\left(\bigcup_{w\in\cL}P_{\aa,\bb}(w)\right) \cap\ZZ^k.
$$
Let $w\in\cL$ and without loss of generality assume $T_1\cup\cdots\cup T_r=\{m,m+1,\ldots,k\}$ for some $m\geq 1$. If $m=1$, then $P_{\aa,\bb}\cap \ZZ^k$ only consists of the set $\cL$. In particular, it is finite and there is nothing to show. Assume $m\geq 2$. In this case, by the same reasoning as above for $P_\aa^{(i)}$, we have
$$
P_{\aa,\bb}(w)=\left\{x=(x_1,\ldots,x_{m-1})\times w\in \RR^{m-1}\times \{w\}~\mid~\sum_{j\in T}x_j\geq \widetilde{a}_T \mbox{ for all } \varnothing\neq T\subseteq [m-1]\right\},
$$
where $\widetilde{a}_T=\max\{a_{T\cup S}-\sum_{i\in S}w_i~\mid~\varnothing\subseteq S\subseteq\{m,m+1,\ldots,k\}\}$. Finally, it follows from the case $U=\varnothing$ that the polyhedron $P_{\aa,\bb}(w)$ can be decomposed into a union of disjoint orthants. This finishes the proof.
\end{proof}

\begin{cor}
    \label{cor:counting}
If $P_{\aa,\bb}$ is not empty then the cardinality of the set
$$
P_{\aa,\bb} \cap \ZZ^k \cap \{(x_1,\ldots,x_k)\in \RR^k \; \mid \; x_1 + \cdots + x_k = n\} 
$$
is given by a polynomial in $n$ whose degree is at most $k-1$ whenever $n$ is a sufficiently large integer. 
\end{cor}

\begin{proof}
To show the statement it is enough to show it for each orthant in the decomposition of \Cref{prop: disjoint cone decomposition}. For each orthant $C$, the integer points in the intersection of $C$ with the hyperplane $\{(x_1,\ldots,x_k)\in \RR^k \; \mid \; x_1 + \cdots + x_k = n\}$ are (up to permuting indices) of the form:
\begin{equation}\label{eq: count}
\{(x_1,\ldots,x_m)\in\mathbb{Z}^m~\mid ~ x_i\geq c_i\text{ for } 1\leq i\leq m,\;x_1+\cdots +x_m=n-a\}\times \{w\}
\end{equation}
for some $w\in \ZZ^{k-m}$ and integers $a$ and $c_i$ for $1\leq i\leq m$. The cardinality of the set in \eqref{eq: count} is given by $\binom{n-a+m-\sum_{i=1}^mc_i}{m-1}$ and for $n\geq a+\sum_{i=1}^mc_i-m$, this is a polynomial in $n$ of degree $m-1$. Since $m\leq k$ and the lattice points in $P_{\aa,\bb}$ are given as the disjoint union of the lattice points in each orthant of the decomposition, the claim follows.
\end{proof}

We note that after shifting the apex of each orthant of the cone decomposition of $P_{\aa,\bb}$ to the origin, the lattice points counted in each orthant can be seen as the lattice points of dilations of simplices whose vertices are the unit vectors.
 
\section{Counting Minimal Generators} \label{sec: counting}
The goal of this section is to study the number of minimal generators of Alexander duals to any $\Sym$-invariant chain  $(I_n)_{n \in \NN}$ of squarefree monomial ideals. We show that the number of $\Sym(n)$-orbits of monomials that minimally generate  $I_n^{\vee}$ grows eventually as a polynomial in $n$.

We continue to use the notation from Section \ref{sec: generalCase}. Our starting point is \Cref{thm:min gens in general}, which provides a description of the minimal generating set via the set $\cG (I_n^{\vee})$. In the following, we use  
\[
\cP^{\langle 2\rangle}\coloneqq\cP^2- \{((\cF,\cC),(\cF,\cC))~\mid~(\cF,\cC)\in \cP\}
\]
to denote the set of pairs of distinct elements of $\cP$. The next result, though technical, is one of the key steps to achieve our goal.

\begin{prop}\label{lem:InclusionExclusion} If $n \ge m$, one has that
\begin{align*}
    \#\cG(I_n^\vee)=   \sum_{\varnothing \neq \cA \subseteq \cP} (-1)^{\# \cA + 1} \  \#\bigcap_{(\cF,\cC)\in \cA}\cG_{\cF,\cC}(n) 
    %
    + 
    \sum_{\varnothing \neq \cA \subseteq \cP^{\langle 2\rangle}} \sum_{C\subseteq \cA}(-1)^{\# \cA + \#C + 1} \ \# \cG_{\cA, C}(n), 
    \end{align*}
where
\[
\cG_{\cA, C}(n) = 
\bigcap_{((\cF,\cC),(\cF',\cC'))\in C}(\cG_{\cF,\cC}(n)\cap \cG_{\cF',\cC'}(n))\cap \bigcap_{((\cF,\cC),(\cF',\cC'))\in \cA- C}(\cG_{\cF,\cC}(n)\cap (I_{\cF',\cC'}(n))
\]
\end{prop}

\begin{proof}
For counting the number of elements of $\cG(I_n^\vee)$, it is convenient to rewrite the minimal generating set given in \Cref{thm:min gens in general}. 
Note that a monomial $\xx^B$ with $B\in \cG_{\cF,\cC}(n)$ is \emph{not} a minimal generator of $I_n^\vee$ if and only if there exists $(\cF',\cC')\in \cP - \{(\cF,\cC)\}$ with
\begin{equation*}
    \xx^B\in \ideal{ \xx^A~\mid~ A\in \cG_{\cF',\cC'}(n) }- \{\xx^A~\mid~A\in \cG_{\cF',\cC'}\}.
\end{equation*}
It follows that 
\begin{equation*}
    \cG(I_n^\vee)=\bigcup_{(\cF,\cC)\in \cP}\cG_{\cF,\cC}(n) \ - \bigcup_{\substack{((\cF,\cC),(\cF',\cC'))\in \cP^2\\ (\cF,\cC)\neq (\cF',\cC')}}\cG_{\cF,\cC}(n)\cap \left(I_{\cF',\cC'}(n)- \cG_{\cF',\cC'}(n)\right).
\end{equation*}
Applying the principle of inclusion-exclusion to the first union yields the first sum in the claimed expression for $\#\cG(I_n^\vee)$. Similarly, by the inclusion-exclusion principle, the cardinality of the second union equals
\begin{equation}\label{eq:secondUnion}
\sum_{\varnothing \neq \cA \subseteq \cP^{\langle 2\rangle}} (-1)^{\# \cA + 1} \ 
    \#\bigcap_{((\cF,\cC),(\cF',\cC'))\in\cA}\left(\cG_{\cF,\cC}(n)\cap (I_{\cF',\cC'}(n)-\cG_{\cF',\cC'}(n))\right).
\end{equation}
It remains to control the cardinalities of the above intersections. For this, we use the following inclusion-exclusion-like fact: Given finite sets $A_1,\ldots,A_i$ and $B_1,\ldots,B_i$ with $B_j\subseteq A_j$ for $1\leq j\leq i$ one has
$$
\#\bigcap_{j=1}^i(A_j- B_j)=\sum_{S\subseteq [i]}(-1)^{\#S} \ \#\left(\bigcap_{j\in S}B_j\cap\bigcap_{j\in [i]- S}A_j\right).
$$
Since $\cG_{\cF',\cC'}(n)\subseteq  I_{\cF',\cC'}(n)$ and 
$$
\cG_{\cF,\cC}(n)\cap (I_{\cF',\cC'}(n)-\cG_{\cF',\cC'}(n))=\left(\cG_{\cF,\cC}(n)\cap I_{\cF',\cC'}(n)\right)-\left(\cG_{\cF,\cC}(n)\cap \cG_{\cF',\cC'}(n)\right), 
$$
we conclude
\begin{align*}
    &\#\bigcap_{((\cF,\cC),(\cF',\cC'))\in\cA}\left(\cG_{\cF,\cC}(n)\cap (I_{\cF',\cC'}(n)-\cG_{\cF',\cC'}(n))\right)\\
    & = \sum_{C\subseteq \cA}(-1)^{\#C} \# \left(\bigcap_{((\cF,\cC),(\cF',\cC'))\in C}(\cG_{\cF,\cC}(n)\cap \cG_{\cF',\cC'}(n))\cap \bigcap_{((\cF,\cC),(\cF',\cC'))\in \cA- C}(\cG_{\cF,\cC}(n)\cap (I_{\cF',\cC'}(n))\right).
\end{align*}
Substituting this expression in Equation \eqref{eq:secondUnion} yields the claim.
\end{proof}

\begin{remark}
    \label{rem:enough to count}
By  \Cref{lem:InclusionExclusion}, in order to show that $\#\cG(I_n^\vee)$ is a polynomial in $n$ for $n\gg 0$, it suffices to establish that for any  subsets $U$ and $V$ of $\cP$, we have that 
\[ 
\#\left(\bigcap_{(\cF,\cC)\in U}\cG_{\cF,\cC}(n)\cap \bigcap_{(\cF',\cC')\in V}I_{\cF',\cC'}(n)\right )
\]
    is a polynomial in $n$ for $n\gg 0$.
\end{remark}

To analyze the cardinalities of the sets in \Cref{rem:enough to count}, we need to describe the involved intersections.  This requires some preparations. 
    
 It is useful to extend the definitions of $\cG_{\cF, \cC}$ and ideal $I_{\cF,\cC}(n)$ (see in \Cref{def: generalGenSet} and \Cref{thm:min gens in general}) which were originally for pairs $(\cF, \cC) \in \cP$ to more general pairs.

\begin{notation}\label{not: generalized} Denote by $\cQ$ the set of pairs $(\cF, \cC)$, where $\cC$ is an antichain in $2^{[c]}$ and $\cF = (\ell_S \in \NN \; \mid \;  S\in F)$ for some nonempty subset $F$ of $2^{[c]}$ satisfying $F\cap \langle \cC\rangle = \varnothing$. In particular, $\cP\subseteq \cQ$.

Any such $\cF$ corresponds to a unique matrix $M_{\cF} = \left[\Ast_{S \in F} A_S(\ell_S) \right]$ whose support is $F$. As above, we set for an integer $n \ge \sum_{S \in F} \ell_S$, 
\begin{align*}
\cG_{\cF,\cC}(n)
&=\left\{ M_{\cF} \ast \left[ \Ast_{S\in\cC} A_S(\ell_S)\right] \; \mid \; \sum_{S\in\cC}\ell_S = n - \sum_{S\in F}\ell_S \text{ and } \ell_S\geq 0 \text{ if } S\in \cC\right\} 
\end{align*}
and 
\begin{equation*}
   I_{\cF,\cC}(n)=\{B \in \{0, 1\}^{c \times n} \; \mid \;  \xx^B\in  \langle \xx^A \; \mid \; A\in \cG_{\cF,\cC}(n) \rangle \subset R_n\}.
\end{equation*}  
\end{notation} 

We now describe the intersections arising in \Cref{rem:enough to count} in the case where $V = \varnothing$. We begin with the special case that identifies matrices which appear simultaneously in two different sets of the form $\cG_{\cF,\cC}(n)$. 

\begin{lemma}
     \label{lem: almostMinimalGenSet} 
Consider $(\cF,\cC) \neq (\cF', \cC')$ in $\cQ$. Set 
$$
M_\cF = \left[ \Ast_{T\in F} A_T(j_T)\right] \quad \text{and}\quad M_{\cF'} = \left[ \Ast_{T\in F'} A_T(\ell_T)\right].
$$
Then one has that $\cG_{\cF,\cC}(n)\cap \cG_{\cF',\cC'}(n)\neq \varnothing$ for all $n > \sum_{T \in F} j_T + \sum_{T \in F'} \ell_T$ if and only if the following three conditions are satisfied:
\begin{enumerate}[(i)]
    \item $F\subseteq F'\cup \cC'$ and $F'\subseteq F\cup\cC$;
    \item If $T\in F\cap F'$, then $j_T = \ell_T$; and
    \item $\cC\cap\cC' \neq \varnothing$.
\end{enumerate}
In this case,  we have $\cG_{\cF,\cC}(n)\cap \cG_{\cF',\cC'}(n)=\cG_{\tilde{\cF},\cC\cap\cC'}$,
where $\tilde{\cF} :=(i_T~\mid~T\in F\cup F')$ with $i_T=j_T$ if $T\in F$ and $i_T=\ell_T$ if $T\in F'$.
\end{lemma}

We note that $\tilde{\cF}$ is well-defined due to $(ii)$.

\begin{proof} 
For ease of notation, set  $\tilde{\cC} = \cC \cap \cC'$.  

First we show necessity of Conditions $(i)$--$(iii)$. 
Consider any matrix  $B = \left[\Ast_{T\in 2^{[c]}} A_T(i_T)\right] \in \cG_{\cF,\cC}(n)\cap \cG_{\cF',\cC'}(n)$. Then $F\cup F' \subseteq \Supp B \subseteq (F\sqcup \cC)\cap (F'\sqcup \cC')$, giving necessity of Condition $(i)$. We also have the following facts:
\begin{enumerate}[(a)]
    \item If $T\in F\cap F'$, then $i_T = \ell_T = j_T$. This is exactly Condition $(ii)$.
    \item If $T\in F\cap \cC'$, then $i_T = j_T$,
    \item If $T\in F'\cap \cC$, then $i_T = \ell_T$ ,
    \item If $T\in \cC-(F'\cup \cC')$ or if $T\in \cC' - (\cC\cup F)$, then $i_T = 0$ , and so, 
    $$
    \sum_{T\in \cC\cap \cC'} i_T = n - \sum_{T\in F} j_T - \sum_{T\in F'\cap \cC} \ell_T.
    $$
    \end{enumerate}
Since we assume $n >  \sum_{T \in F} j_T + \sum_{T \in F'} \ell_T \ge \sum_{T\in F} j_T + \sum_{T\in F'\cap\cC} \ell_T$, Fact (d) forces $\cC\cap \cC' \neq \varnothing$. Moreover, it follows that  $B$ is in $\cG_{\tilde{\cF},\tilde{\cC}}(n)$. 

To show sufficiency of $(i)$--$(iii)$, observe that Condition $(i)$ implies 
\[
F \sqcup (F' \cap \cC) = (F \cap F') \sqcup (F \cap \cC') \sqcup (F' \cap \cC) = F' \sqcup (F \cap \cC'). 
\]
Together with Condition $(ii)$, this gives 
\[
M_\cF \ast \left[\Ast_{T\in F'\cap \cC}A_T(\ell_T)\right] = M_{\cF'} \ast \left[\Ast_{T\in F\cap \cC'}A_T(j_T)\right]. 
\]
It follows that $\cG_{\tilde{\cF},\tilde{\cC}}(n)$ is contained in $\cG_{\cF,\cC}(n)\cap \cG_{\cF',\cC'}(n)$. Moreover, $\cG_{\tilde{\cF},\tilde{\cC}}(n)$ is not empty if $n > \sum_{T \in F} j_T + \sum_{T \in F'} \ell_T$. 
\end{proof} 


The previous result easily generalizes.  
\begin{cor} \label{cor: easy intersection}
Consider any subset $U$ of $\cP$. If $n > \# U \cdot s (m-1)$, then the set $\bigcap_{(\cF,\cC) \in U} \cG_{\cF,\cC}(n)$ is either empty or equal to $\cG_{\tilde{\cF},\tilde{\cC}}(n)$ for some $(\tilde{\cF},\tilde{\cC}) \in \cQ$ with $\tilde{\cC} = \bigcap_{(\cF,\cC)\in U}\cC \neq \varnothing$. 

In the latter case,  
\[
\# \bigcap_{(\cF,\cC) \in U} \cG_{\cF,\cC}(n) = \binom{\# \tilde{\cC} + n - k - 1}{-1 +  \# \tilde{\cC}} 
\]
is a polynomial in $n$ of degree $-1 + \#\tilde{\cC}$, where $k$ is the number of columns of the matrix $M_{\tilde{\cF}}$. 
\end{cor}

\begin{proof} 
We use induction on $\# U$ to prove the assertion and to prove that $M_{\tilde{\cF}}$ has at most $\# U \cdot s\cdot (m-1)$ columns. 
If $U$ is empty, there is nothing to show. Assume $U = \{(\cF, \cC)\}$. By the definition of $\cP$, the pair $(\cF, \cC)$ is determined by an 
$s$-tuple $\jJ = (J_1,\dots, J_s)$  of proper order ideals of $2^{[c]}$. Since every matrix $A_i$ has at most $m$ nonzero columns, the System of Inequalities \eqref{eq: generalIneqs} shows that any solution $L = (\ell_S \; \mid \; S\in \overline{J}_1\cup\dots \cup \overline{J}_s)$ satisfies $\sum_{\ell_S \in \overline{J}_i} \ell_S < m$. It follows that $M_{\cF}$ has at most $s (m-1)$ columns. This completes the argument if $\# U = 1$. 

Assume $\# U \ge 2$. Choose some $(\cF', \cC') \in U$. By the induction hypothesis,  \Cref{lem: almostMinimalGenSet} is applicable to the sets $\cG_{\cF', \cC'}(n)$ and $\bigcap_{(\cF,\cC) \in U - (\cF', \cC') } \cG_{\cF,\cC}(n)$. It shows that $\bigcap_{(\cF,\cC) \in U} \cG_{\cF,\cC}(n)$ is either empty or equal to $\cG_{\tilde{\cF},\tilde{\cC}}(n)$ for some $(\tilde{\cF},\tilde{\cC}) \in \cQ$ with $\tilde{\cC} = \bigcap_{(\cF,\cC)\in U}\cC \neq \varnothing$. As by \Cref{lem: almostMinimalGenSet}, the support of $\tilde{\cF}$ is the union of the supports of all $\cF$ with  $(\cF,\cC)\in U$, this completes the induction step.

If $\cG_{\tilde{\cF},\tilde{\cC}}(n)$ is not empty, its cardinality is $\binom{\# \tilde{\cC} + n - k - 1}{-1 +  \# \tilde{\cC}}$. 
\end{proof}


In order to describe an arbitrary intersection as specified in \Cref{rem:enough to count}, we need a result in the spirit of Lemma \ref{lem:orbit intersect one-orbit case} for the one-orbit case. It may also be used to identify   redundancies in the generating set in \Cref{cor:smaller gen set}.

\begin{lemma}[Membership Lemma] 
    \label{lem: membership} 
Let  $(\cF, \cC) \in \cQ$ be any pair, where  $M_{\cF} = \left[\Ast_{T\in F} A_T(\ell_T)\right]$ (and so $F = \Supp(M_{\cF})$).  
For any integer $\lambda \ge 0$, set $n = \lambda + \sum_{T \in F} \ell_T$ and consider any $0-1$-matrix  $B = \left[\Ast_{T\in 2^{[c]}} A_T(j_T)\right]$ in $\ZZ^{c\times n}$. 
%
%
Then $\xx^B$ is in the ideal $I_{\cF,\cC}(n)$ if and only if the following two conditions are satisfied:
\begin{enumerate}[(i)]
    \item For every nonempty subset $E\subseteq F$ with $\cC\not \subseteq \langle E\rangle$, one has
    $$
    \sum_{T\in \langle E \rangle} j_T \geq \sum_{T\in F\cap\langle E\rangle} \ell_T.
    $$
    \item For every subset $E\subseteq F\cup \cC$ with $\cC\subseteq E$, one has 
    $$
    \sum_{T\in \langle E\rangle} j_T \geq \sum_{T\in \langle E\rangle} \ell_T = \lambda +  \sum_{T\in F\cap \langle E\rangle} \ell_T.
    $$
\end{enumerate}
\end{lemma}

\begin{proof}
Applying the Divisibility Lemma \ref{lem: divisibility up to symmetry} to the case where $J = \langle E\rangle$ gives $\sum_{T\in \langle E\rangle} j_T \geq \sum_{T\in \langle E\rangle} \ell_T$, which shows 
necessity of  the conditions in $(i)$ and $(ii)$. 
Notice that the obtained inequality  is stronger than the condition in $(i)$ for the case where $E\subseteq F$.

To prove that Conditions $(i)$ and $(ii)$ are sufficient, we proceed by induction on $\abs{\cC}$. Assume that $\#\cC = 1$. By Lemma \ref{lem: divisibility up to symmetry}, we have to show that $\sum_{T\in \langle E\rangle} j_T \geq \sum_{T\in \langle E\rangle} \ell_T$
for any nonempty subset $E$ of $F \cup \cC$. 
We consider several cases. If $\cC\subseteq E$, then the desired inequality is true by Condition (ii). Suppose $\cC\not \subseteq E$. 
 Since $\#\cC = 1$ and $E \subset F \cup \cC$ this implies $E \subseteq F$. If $\cC\subseteq \langle E\rangle$, the divisibility condition of Lemma \ref{lem: divisibility up to symmetry} for $J = \langle E\rangle$ is the same as condition $(ii)$ for $E\cup \cC$. 
 If $\cC\not \subseteq \langle E\rangle$, then the assumption $\#\cC = 1$ implies that $\cC\cap \langle E\rangle = \varnothing$, and so the divisibility condition for $J = \langle E\rangle$ is true by Condition $(i)$.
%
%

Now suppose that $\#\cC \geq 2$. Choose any $S \in \cC$ and set
$$
\ell_S = \min \big \{\lambda, \sum_{T\in \langle S\rangle - \langle F\cup \cC- S\rangle} j_T \big \}.
$$
By the induction hypothesis, there is a fixed tuple $(\ell_T)_{T\in \cC- S}$ such that $\xx^B$ is divisible by  $\xx^{\tilde{B}}$ with 
$$
\tilde B = M_F \ast \left[\Ast_{T\in \cC-S} A_T(\ell_T)\right], 
$$
where $\sum_{T\in \cC-S} \ell_T = \lambda -\ell_S$. 
By Divisibility Lemma \ref{lem: divisibility up to symmetry}, this means that for every subset $\varnothing \neq E \subseteq F\cup \cC- S$, one has
\begin{equation}\label{eq: ineqMembershipLemma1}
    \sum_{T\in \langle E\rangle - S} \ell_T \leq \sum_{T\in \langle E\rangle} j_T.
\end{equation}
We claim that $\xx^B$ is divisible by $\xx^{B'}$ up to symmetry with 
$$
B' = M_{\cF} \ast \left[\Ast_{T\in \cC} A_T(\ell_T)\right] = \tilde B \ast A_S(\ell_S),
$$
that is, for every $\varnothing \neq E \subseteq F\cup \cC$, one has
\begin{equation}\label{eq: ineqMembershipLemma2}
    \sum_{T\in \langle E \rangle} \ell_T \leq \sum_{T\in \langle E\rangle} j_T.
\end{equation} 

To this end we consider several cases. 
If $S\in \langle F\cup \cC-S\rangle$, then $\ell_S = 0$ by definition, and so $B' = \tilde B$. Therefore, assume that $S\not\in \langle F\cup \cC - S\rangle$. For every subset $E\subset F\cup \cC$ with $S\notin E$, this implies that $S\notin \langle E\rangle$, so  Inequality (\ref{eq: ineqMembershipLemma2}) follows from Inequality (\ref{eq: ineqMembershipLemma1}).

Consider now any subset $E\subseteq F\cup \cC$ with $S\in E$. By definition of $\ell_S$, one has
\begin{equation}\label{eq: lambdaIneq}
    \ell_S \quad \leq \quad \sum_{T\in \langle S\rangle - \langle F\cup \cC - S\rangle} j_T \quad \leq \quad \sum_{T\in \langle S\rangle - \langle E-S\rangle} j_T.
\end{equation}
Moreover, Inequality (\ref{eq: ineqMembershipLemma1}) applied to $E - S$ gives
\begin{equation} \label{eq: E-S ineq}
    \sum_{T\in \langle E-S\rangle - S} \ell_T \; \leq \sum_{T\in \langle E-S\rangle} j_T.
\end{equation}
Notice the decomposition 
\begin{equation}
    \label{eq:decompose}
\langle E\rangle = \langle E-S\rangle \sqcup (\langle S\rangle - \langle E - S\rangle). 
\end{equation}
Thus, adding Inequalities (\ref{eq: lambdaIneq}) and (\ref{eq: E-S ineq})  yields
\begin{equation}\label{eq: almostThere}
\ell_S + \sum_{T\in \langle E-S\rangle - S} \ell_T \leq \sum_{T\in\langle E\rangle} j_T.
\end{equation}
Since we chose $S$ in $\cC$, we get $F\cap \langle S\rangle \subseteq F\cap \langle \cC\rangle = \varnothing$ by assumption, and so Equation \eqref{eq:decompose} implies 
\begin{equation}\label{eq: FcapE}
    F\cap \langle E\rangle = (F\cap \langle E - S\rangle ) \sqcup (F\cap (\langle S\rangle - \langle E - S\rangle )) = F\cap \langle E - S\rangle .
\end{equation}
Using Equation \eqref{eq:decompose} again, we obtain 
\[
\cC \cap \langle E \rangle = (\cC \cap \langle E - S \rangle)   \sqcup  \big ( \cC \cap ( \langle S \rangle -  \langle E - S ) \rangle \big ). 
\]
We also have that $\cC \cap \big ( \langle S \rangle -  \langle E - S \rangle \big ) \subseteq \cC \cap \langle S \rangle = \{S\}$ because $\cC$ is an antichain that contains $S$. It follows that 
\[
(\cC \cap \langle E \rangle) - S \subseteq \cC \cap \langle E - S \rangle, 
\]
which implies 
\[
(\cC \cap \langle E \rangle) - S = ( \cC \cap \langle E - S \rangle) - S.  
\]
Combining this with Equation \eqref{eq: FcapE} and the fact that $F \cap \cC = \varnothing$ by assumption, and so $S \notin F$, we conclude 
\begin{align*}
\big ( (F \cup \cC) \cap \langle E - S \rangle \big )  - S 
& = (F \cap \langle E - S \rangle)  \sqcup \big (  (\cC \cap \langle E - S \rangle) - S. \big ) \\
& = (F \cap \langle E \rangle)  \sqcup \big (  (\cC \cap \langle E  \rangle) - S \big ) \\
& = (F \cup \cC) \cap \langle E \rangle )  - S.  
\end{align*}
This shows that 
\[
\ell_S + \sum_{T\in \langle E-S\rangle - S}  \ell_T = \ell_S + \sum_{T\in \langle E \rangle - S} \ell_T =  \sum_{T\in \langle E\rangle } \ell_T . 
\]
Hence, Inequality \eqref{eq: almostThere} gives
$$
\sum_{T\in \langle E\rangle} \ell_T \leq \sum_{T\in \langle E\rangle} j_T,
$$
which completes the argument.
\end{proof}

We now apply Membership \Cref{lem: membership} to the case where the matrix $B$ belongs to some set $\cG_{\cF', \cC'}(n)$ with $(\cF', \cC') \in \cQ$.

\begin{cor} 
    \label{cor: ADmembershipLemma}
Consider any pairs $(\cF, \cC), (\cF', \cC') \in \cQ$ with  $M_\cF = \left[\Ast_{T\in F} A_T(j_T)\right]$ and $M_{\cF'} = \left[\Ast_{T\in F'} A_T(\ell_T)\right]$. 
Then a monomial  $\xx^B$ with $B = M_\cF \ast \left[\Ast_{T\in \cC} A_T(j_T)\right] \in G_{\cF, \cC} (n)$ belongs to the ideal $I_{\cF', \cC'} (n)$  if and only if the following two  conditions are satisfied:
%
\begin{enumerate}[(i)]
    \item For every  nonempty subset $E \subseteq F'$ with $\cC' \not\subseteq \langle E \rangle$, one has
    $$
    \sum_{T\in \cC \cap \langle E\rangle} j_T \geq \sum_{T\in F'\cap \langle E\rangle} \ell_T - \sum_{T\in F\cap \langle E\rangle} j_T.
    $$
    \item For any subset $E \subseteq F'\cup \cC'$ with $\cC' \subseteq E$, one has
    $$
    \sum_{T\in \cC-\langle E\rangle} j_T \leq \sum_{T\in F'-\langle E\rangle} \ell_T - \sum_{T\in F-\langle E\rangle} j_T.
    $$
\end{enumerate}
\end{cor}

\begin{proof}
There are two cases to consider:
\begin{itemize}
    \item If $\varnothing \neq E \subseteq F'$ satisfies $\cC' \not \subseteq \langle E\rangle$, then Condition (i) from Membership  Lemma \ref{lem: membership} gives
    $$
    \sum_{T\in F\cap\langle E\rangle} j_T + \sum_{T\in \cC\cap\langle E\rangle} j_T \geq \sum_{T\in F'\cap \langle E\rangle} \ell_T
    $$
    which is equivalent to Condition (i) in the claim.
    \item If $\cC' \subseteq E\subseteq F'\cup \cC'$, then Condition (ii) in Membership Lemma \ref{lem: membership} yields
    \begin{equation}\label{eq: claimiiADmembership}
    \sum_{T\in (F\cup \cC)\cap \langle E\rangle} j_T \geq \sum_{T\in F'\cap \langle E\rangle} \ell_T + n- \sum_{T\in F'} \ell_T = n - \sum_{T\in F'-\langle E\rangle} \ell_T
    \end{equation}
    where the lefthand side  is equal to
    $$
    n - \sum_{T\in F-\langle E\rangle} j_T - \sum_{T\in \cC - \langle E\rangle} j_T.
    $$
    Therefore, Inequality (\ref{eq: claimiiADmembership}) is equivalent to Condition (ii) in the claim.    
\end{itemize}
\end{proof}

We summarize the discussion above to describe the intersections appearing in \Cref{rem:enough to count}. 

\begin{prop}
   \label{prop: arbitrary intersection}
For any subsets $U, V$ of $\cP$ and any integer $n \gg 0$, consider the set 
\[
\cI = \bigcap_{(\cF,\cC)\in U}\cG_{\cF,\cC}(n)\cap \bigcap_{(\cF',\cC')\in V}I_{\cF',\cC'}(n). 
\]
If $\bigcap_{(\cF,\cC)\in U}\cG_{\cF,\cC}(n)$ is empty, then $\cI$ is empty. 
Otherwise,   $\bigcap_{(\cF,\cC)\in U}\cG_{\cF,\cC}(n) = \cG_{\tilde{\cF},\tilde{\cC}}(n)$ for some 
$(\tilde{\cF},\tilde{\cC}) \in \cQ$ with $\tilde{\cC} = \bigcap_{(\cF,\cC)\in U}\cC \neq \varnothing$. Moreover, a matrix $B$  is in the set $\cI$ if and only if
    $B= M_{\tilde{\cF}} \ast \left[\Ast_{T \in \tilde{\cC}} A_T(j_T)\right] \in \cG_{\tilde{\cF},\tilde{\cC}}(n)$ and, for every $(\cF', \cC') \in V$, 
    the following two  conditions are satisfied:
\begin{enumerate}[(i)]
    \item For every  nonempty subset $E \subseteq F'$ with $\cC' \not\subseteq \langle E \rangle$, one has
    $$
    \sum_{T\in\tilde{\cC} \cap \langle E\rangle} j_T \ge \sum_{T\in F'\cap \langle E\rangle} \ell_{T, \cF'} - \sum_{T\in \tilde{\cF} \cap \langle E\rangle} j_T ; 
    $$
    \item For any subset $E \subseteq F'\cup \cC'$ with $\cC' \subseteq E$, one has
    $$
    \sum_{T\in\tilde{\cC} -\langle E\rangle} j_T \le \sum_{T\in F'-\langle E\rangle} \ell_{T, \cF'} - \sum_{T\in \tilde{\cF} - \langle E\rangle} j_T, 
    $$
\end{enumerate}
where $M_{\cF'} = [\Ast_{T \in \Supp (\cF')} A_T (\ell_{T, \cF'})]$. 
\end{prop} 

\begin{proof}
The assertion about $\bigcap_{(\cF,\cC)\in U}\cG_{\cF,\cC}(n)$ it true by \Cref{cor: easy intersection}. 
Now, the characterization of the monomials in $\cI$ follows from \Cref{cor: ADmembershipLemma}. 
\end{proof}

We are ready to establish the main result of this section, which is  \Cref{thm:intro-gen-degrees} of the Introduction. We restate an equivalent version here for the convenience of the reader. 

\begin{theorem} 
    \label{thm:gens from intro} 
Consider nonzero ideals $I_n= \langle \Sym(n)(\xx^{A_1},\ldots,\xx^{A_s}) \rangle$, where $A_i\in \{0,1\}^{c\times m}$ for $1\leq i\leq s$ and $n\geq m$. 
The number of $\Sym(n)$-orbits needed to minimally generate $I_n^\vee$ is a polynomial in $n$ whose degree is at most 
$ \binom{c}{\lfloor \frac c2 \rfloor} -1$ if  $n$ is sufficiently large. 

Moreover, there is a linear function $d \colon \NN \to \NN$ such that, for $n \gg 0$, the least degree of a minimal generator of $I_n^{\vee}$ equals $d(n)$. The number of $\Sym(n)$-orbits needed to produce the minimal generators of $I_n^{\vee}$ of degree $d(n)$ is given by a polynomial in $n$. 
\end{theorem} 

\begin{proof} 
In order to establish the first statement, using  \Cref{thm:min gens in general}, we have to determine the cardinality of the set $\cG (I_n^{\vee})$ of matrices. \Cref{lem:InclusionExclusion} shows that $\# \cG (I_n^{\vee})$ is given by a polynomial in $n$ for $n \gg 0$ if this is true for the cardinality of any intersection $\bigcap_{(\cF,\cC)\in U}\cG_{\cF,\cC}(n)\cap \bigcap_{(\cF',\cC')\in V}I_{\cF',\cC'}(n)$ with subsets $U, V$ of $\cP$ (see \Cref{rem:enough to count}). Fix $U$ and $V$ and assume that the corresponding intersection, denoted $\cI$, is not empty. In this case, \Cref{prop: arbitrary intersection} shows that $\cI = \cG_{\tilde{\cF},\tilde{\cC}}(n) \cap \bigcap_{(\cF',\cC')\in V}I_{\cF',\cC'}(n)$ for some 
$(\tilde{\cF},\tilde{\cC}) \in \cQ$.   

For any $(\cF', \cC') \in V$, consider subsets $Y_{\cF', \cC'} = \{ E \subseteq F' \; \mid \; E \neq \varnothing, \ \cC'  \not\subseteq \langle E \rangle\}$ and $Z_{\cF', \cC'} = \{E \subseteq F' \cup \cC'  \; \mid \; \cC' \subseteq E \}$. 
Writing $M_{\tilde{\cF}} = [\Ast_{T \in \tilde{F}} A_T (j_T)]$ and  $M_{\cF'} = [\Ast_{T \in F' } A_T (\ell_{T, \cF'})]$ with $\tilde{F} = \Supp (\tilde{\cF})$ and $\tilde{F'} = \Supp (\tilde{\cF'})$, define for any subset $D \subseteq \tilde{\cC}$ an integer $a_D$ by 
\[
a_D = \begin{cases}
0 & 
\begin{minipage}{6.9cm}
$\text{ if } D \neq \tilde{\cC} \cap \langle E \rangle \text{ for each }  \\
 E \in Y_{\cF', \cC'} \text{ with $(\cF', \cC') \in V$} $; 
\end{minipage} \\
\displaystyle{ \max \big \{ 0, \sum_{T\in F'\cap \langle E\rangle} \ell_{T, \cF'} - \sum_{T\in \tilde{\cF} \cap \langle E \rangle} j_T } \; \mid \; 
\begin{minipage}{3.9cm}
$D = \tilde{\cC} \cap \langle E \rangle,  E \in Y_{\cF', \cC'} \\
\text{ for some }  (\cF', \cC') \in V$
\end{minipage}
\big \} 
& \text{ otherwise}. 
\end{cases}
\]
For $D \in W = \big \{ D \subseteq \tilde{\cC} \; \mid \; D =  \tilde{\cC} - \langle E \rangle \text{ for some } E \in Z_{\cF', \cC'} \text{ with } (\cF', \cC') \in V \big \}$, set 
\[
b_D = \min \big \{ \sum_{T\in F'\cap \langle E\rangle} \ell_{T, \cF'} - \sum_{T\in \tilde{\cF} \cap \langle E\rangle} j_T  \; \mid \; 
D = \tilde{\cC} - \langle E \rangle,  E \in Z_{\cF', \cC'} 
\text{ for some }  (\cF', \cC') \in V \big \}. 
\]
Assume $\tilde{\cC}$ has $k$ elements $T_1,\ldots,T_k$. Using \Cref{not: polyhedron}, define a polyhedron
\begin{equation*}
P_{\aa,\bb} := \left\{(x_{T_1},\ldots,x_{T_k}) \in \RR^k \; \mid \;  \sum_{T \in D} x_T \ge a_D \text{ for any } D \subseteq \tilde{\cC}, \sum_{T \in D} x_T \le b_D\text{ for any } D \in W \right\}.
\end{equation*}

Then \Cref{prop: arbitrary intersection} shows that $B= M_{\tilde{\cF}} \ast \left[\Ast_{T \in \tilde{\cC}} A_T(j_T)\right]$ is in $\cI$ if and only if 
\[
(j_{T_1},\ldots,j_{T_k}) \in P_{\aa,\bb}  \cap \big \{ (j_{T_1},\ldots,j_{T_k}) \in \NN_0^k \; \mid \;  j_{T_1} + \ldots + j_{T_k} =  n - \sum_{T \in \tilde{F}} \ell_T \big \}. 
\]
By \Cref{cor:counting}, the cardinality of the above intersection is given by a polynomial in $n$ if $n \gg 0$. The degree bound also follows because the longest antichain in $2^{[c]}$ has length $ \binom{c}{\lfloor \frac c2 \rfloor}$. So, our first assertion is proven. 
\medskip

We now show the second assertion. 
By \cite[Theorem 7.10]{NR}, the dimension of $\Delta (I_n)$ is eventually given by a linear function in $n$. Denote it by $D(n)$. For $n \gg 0$, the facets of $\Delta (I_n)$ correspond to minimal generators of $I_n^{\vee}$ of least degree, which is 
$d(n) = c n - 1 - D(n)$.  Write $d(n) = an + b$ for integers $a, b$. Note that $0 \le a \le c$. Since $d(n)$ is the minimum degree of a 
generator of $I_n^{\vee}$,   \Cref{cor:smaller gen set} implies that the set of degree $d(n)$ monomials in $I_n^{\vee}$
 is generated by the $\Sym(n)$-orbits of monomials $\xx^B$ with $B$ in the set
 \begin{equation}
    \label{eq:min deg gens}
 \widetilde{\cG} (n) = \bigcup_{(\cF,\cC)\in \cP} \{A = [ \Ast_{T \in F  \cup \cC} A_T(\ell_T)] \in   \cG_{\cF,\cC} (n) \; \mid \; \deg \xx^A  = d(n) \}. 
 \end{equation}

Consider now the degree of a monomial $\xx^A$ with $A = M_{\cF} \ast [ \Ast_{T \in F  \cup \cC} A_T(\ell_T)] $ in a set $\cG_{\cF,\cC} (n)$.  It is $\deg \xx^A = \sum_{T \in F  \cup \cC} \ell_T \cdot \# T$. 
Since the columns of $A$ with support in $F$ are fixed,  $\xx^A$ has minimum degree if and only if $\ell_T = 0$ whenever $\# T >  m_{\cC}$, where $m_{\cC}$ is the minimum cardinality of an element of $\cC$. Hence,  the minimum degree of monomial $\xx^A$ with $A \in \cG_{\cF,\cC} (n)$ is 
\[
\sum_{T \in F} \ell_T \cdot \# T + m_{\cC} (n - \sum_{T \in F} \ell_T) = m_{\cC} n  + \sum_{T \in F} \ell_T \cdot (\# T - m_{\cC}). 
\]
Comparing with Equation \eqref{eq:min deg gens}, it follows that, for $n \gg 0$, the set $\cG_{\cF,\cC} (n)$ contributes a matrix to  $\widetilde{\cG} (n)$ if and  only if $m_{\cC} = a$ and $b = \sum_{T \in F} \ell_T \cdot (\# T - m_{\cC})$. Denote by $\widetilde{\cP}$ the set of pairs $(\cF, \cC) \in \cP$ that satisfy these two conditions. For $(\cF, \cC) \in \widetilde{\cP}$,  consider the set 
\[
 \widetilde{\cG}_{\cF, \cC} (n) = 
\big\{ A = M_{\cF} \ast \left[ \Ast_{S\in\cC, \# S = m_{\cC}  } A_S(\ell_S)\right] \; \mid \; A \in  \cG_{\cF,\cC} (n) \big \}.
\]

By now we have seen that for $n \gg 0$, 
\[
 \widetilde{\cG} (n) = \bigcup_{(\cF,\cC)\in \widetilde{\cP}}  \widetilde{\cG}_{\cF, \cC} (n). 
\]
Thus, the principle of inclusion-exclusion gives  
\[
\#  \widetilde{\cG} (n) = \sum_{\varnothing \neq \cA \subseteq \widetilde{\cP}} (-1)^{\# \cA} \cdot \# \bigcap_{(\cF, \cC) \in \cA}   \widetilde{\cG}_{\cF, \cC} (n). 
\]
By \Cref{cor: easy intersection}, each of the intersections $\bigcap_{(\cF, \cC) \in \cA}   \widetilde{\cG}_{\cF, \cC} (n)$ equals $\cG_{\cF,\cC} (n)$ for some $(\cF, \cC) \in \cQ$, and so its cardinality is given by a polynomial in $n$ if $n \gg 0$. This completes the argument.
\end{proof}

The above result implies \Cref{cor:intro-number min gens} in the Introduction. 

\begin{proof}[Proof of \Cref{cor:intro-number min gens}] 
The minimal generators of $I_n^{\vee}$ are in bijection to the associated prime ideals of $I_n$. 
\end{proof}

Similarly, the second part of \Cref{thm:gens from intro} has the following consequence.

\begin{cor}
If $(I_n)_{n\in\NN}$ is a $\Sym$-invariant chain of squarefree monomial ideals, then the number of $\Sym(n)$-orbits needed to generate the prime components of $I_n$, whose height equals the height of $I_n$, grows polynomially in $n$ eventually. 
\end{cor}

\Cref{thm:gens from intro} states that the polynomial that gives the number of $\Sym(n)$-orbits of monomials needed to generate the Alexander duals $I_n^\vee$ has degree at most one less than the maximum cardinality of an antichain in $2^{[c]}$. As the following examples show, this bound is sharp, but is certainly not always attained. 

\begin{example}\label{ex: polynomialBoundSharp}
(i) Let $c=2$ and consider the matrix $A= \begin{bmatrix}1\\1\end{bmatrix}$, so that $I_n = \ideal{\Sym(n) \cdot x_{1,1}x_{2,1}}$ for each $n\geq 1$. Here we have $k_{\{1,2\}}=1$, and $k_{\{1\}} = k_{\{2\}}=0$. Below we list $M\cG_\cC(n)$ for each antichain $\cC\subset 2^{[c]}$:

\begin{center}
\renewcommand*{\arraystretch}{1.25}
\begin{tabular}{|c|c|c|c|}
    \hline
    $\cC$ & $2^{[c]}- \overline{\ideal{\cC}}$ & $k_\cC$ & $M\cG_\cC(n)$ \\
    \hline \hline
    $\{1\}$ & $\{1,12\}$ & $1$ & $\left\{ \begin{bmatrix} 1 & \cdots & 1 \\ 0 & \cdots & 0 \end{bmatrix}\right\}$ \\
    \hline 
    $\{2\}$ & $\{2,12\}$ & $1$ & $\left\{ \begin{bmatrix} 0 & \cdots & 0 \\ 1 & \cdots & 1\end{bmatrix}\right\}$ \\
    \hline
    $\{1,2\}$ & $\{12\}$ & $1$ & $\left\{ \begin{bmatrix} 1 & \cdots & 1 & 0 & \cdots & 0 \\ 0 & \cdots & 0 & 1 & \cdots & 1 \end{bmatrix} : j_1 + j_2 = n, j_1, j_2 >0\right\}$ \\
    \hline
    $\{12\}$ & $\{1,2,12\}$ & $1$ & $\varnothing$, since $k_{\{1\}} = k_{\{12\}}$ but $\langle 1 \rangle \supset \{12\}$ \\
    \hline
\end{tabular}
\end{center}

Here, the largest antichain $\cC = \{1,2\}$ of $2^{[2]}$ has $M\cG_\cC(n)$ non-empty for $n \gg0$. Indeed, $\#\cG(n)= n+1$, which is a polynomial of degree $1=\#\cC-1$.

Similarly, in \Cref{ex: two-Orbit yet again}, we had $c=3$ and the number of minimal orbit generators was given by $\frac{n^2}{2}+\frac{5}{2}n-4$. So the degree bound is also attained. 

%


(ii) In contrast, let $c=3$ and consider the matrices $A_1 = \begin{bmatrix} 1 & 0 \\ 0 & 1 \\ 0 & 0\end{bmatrix}$ and $A_2 = \begin{bmatrix} 1 & 0 \\ 0 & 0 \\ 1 & 0 \end{bmatrix}$, so that $I_n = \langle \Sym(n) \cdot \{x_{1,1}x_{2,2}, \, x_{1,1}x_{3,1}\}\rangle$ for $n\geq 2$. We list a table of minimal generators for $I_n^\vee$ by antichain. We only list the antichains which contribute at least one minimal generator:

\begin{center}
\renewcommand*{\arraystretch}{1.25}
\begin{tabular}{|c|c|c|}
\hline
$\cC$ & $M\cG_\cC(n)$ & $\# M\cG_\cC(n)$ \\
\hline \hline
$\{1\}$ & $\left\{\begin{bmatrix} 1 & \cdots & 1 \\ 0 & \cdots & 0 \\ 0 & \cdots & 0 \end{bmatrix}\right\}$ & $1$ \\
\hline
$\{3, 12\}$ & $\left\{ \begin{bmatrix}1 & \cdots & 1 & 0 \\
1 & \cdots & 1 & 0 \\
0 & \cdots & 0 & 1\end{bmatrix} \right\}$  & $1$\\
\hline
$\{12, 23\}$ & $\left\{\begin{bmatrix} 1 & \cdots & 1 & 0 & \cdots & 0 \\
1 & \cdots & 1 & 1 & \cdots & 1 \\ 0 & \cdots & 0 & 1 & \cdots & 1 \end{bmatrix} : j_{12}\geq 1, j_{23}\geq 2\right\}$ & $n-2$ (for $n\geq 3$) \\
\hline
$\{23\}$ & $\left\{\begin{bmatrix} 0 & \cdots & 0 \\ 1 & \cdots & 1 \\ 1 & \cdots & 1 \end{bmatrix}\right\}$ & $1$ \\ \hline
\end{tabular}
\end{center}

For $n\geq 3$, the number of minimal generators for $I_n^\vee$ is equal to $n+1$, a linear polynomial in $n$. Thus the degree bound of Theorem \ref{thm:gens from intro} is not attained.
\end{example}

\section{An application to face numbers}
The goal of this section is to provide an application of one of our main results, \Cref{thm:gens from intro}, to chains of simplicial complexes. More precisely, given a $\Sym$-invariant chain $(I_n)_{n\in \mathbb{N}}$ of squarefree monomial ideals $I_n\subseteq R_n=\kk[x_{i,j}~\mid~i\in [c],j\in [n]]$, we are interested in the face numbers of the associated Stanley-Reisner complexes $(\Delta(I_n))_{n\in \mathbb{N}}$. Since by \Cref{prop: symmetricFaces}  $\Delta(I_n)$ is $\Sym(n)$-invariant, for every $0\leq i\leq \dim\Delta(n)-1$ there exist finitely many $i$-dimensional faces $F_1,\ldots,F_t$ of  $\Delta(I_n)$ such that $\Sym(n)\{F_1,\ldots,F_t\}$ equals the set of all $i$-faces of $\Delta(I_n)$, i.e., the set of $i$-faces of $\Delta(I_n)$ consists of only finitely many $\Sym(n)$-orbits. In the following, we let $f_i(n)^{\Sym}$ denote this number. Our goal is to prove that $f_i(n)^{\Sym(n)}$ is a polynomial in $n$ for $n$ sufficiently large. 

Since the chain $(I_n)_{n\in \mathbb{N}}$ stabilizes, as in the previous sections, we can assume that $I_n=(0)$ for $n\leq m$ and $I_n=\Sym(n)\cdot\langle \xx^{A_1},\ldots,\xx^{A_s}\rangle$ for $n\geq m$, where $A_1,\ldots,A_s$ are $0-1$-matrices of size at most $c\times m$. 

The following lemma will be crucial.

\begin{lemma}\label{lem: allFacetsAppear}
Fix $j\geq 0$. Then the number of $\Sym(n)$-orbits of facets of dimension $j$ in $\Delta(I_n)$ is eventually constant.
\end{lemma}

\begin{proof}
Let $n\geq \max(m,j+1)$ and let $F$ be a facet of $\Delta(I_{n+1})$ with $\dim F=j$. Then, by \Cref{fact: alexDualComplements} $\frac{\textbf{X}}{F}\in I_{n+1}^\vee$. Let $A$ be the exponent matrix of $\frac{\textbf{X}}{F}$. It has $c(n+1)-(j+1)$ entries equal to $0$. Since $n\geq j+1$, there has to exist a column of $A$ with support $[c]$. We can further assume that this is the last column of $A$. Let $B$ be the matrix obtained from $A$ by removing the last column. Then \Cref{fact:mon in dual} implies that $\xx^B\in I_{n}^\vee$, i.e., $F\in\Delta(I_n)$. It remains to show that $F$ is a facet of $\Delta(I_n)$. Otherwise, there exists some $G\in \Delta(I_n)$ with $F\subsetneq G$. By \Cref{lem: subcomplex} it follows that $G\in \Delta(I_{n+1})$, which yields a contradiction.
\end{proof}

We now formulate our main result of this section.

\begin{theorem}\label{prop: fVectorGeneral}
Given a $\Sym$-invariant chain of squarefree monomial ideals $(I_n)_{n\in \mathbb{N}}$ and fixing an integer $j\geq 0$,  the number $f_j(n)^{\Sym}$ of $\Sym(n)$-orbits that form the $j$-faces of $\Delta(I_n)$ is a polynomial in $n$ when $n$ is sufficiently large.
\end{theorem}

\begin{proof}
We first assume that $\Delta(I_n)$ is pure for $n\gg 0$. 
We denote by $\Delta^{(j)}(I_n)$ the $j$-skeleton of  $\Delta(I_n)$, i.e., the simplicial complex consisting of all faces of $\Delta(I_n)$ of dimension at most $j$. Algebraically, this means that $I_{\Delta^{(j)}(I_n)}=I_{\Delta(I_n)}+J_{j,n}$, where $J_{j,n}\subset R_n$ is the ideal generated by all squarefree monomials of degree $j+2$. It is easy to see that $(J_{j,n})_{n\geq j+2}$ is a $\Sym$-invariant chain, hence so is $(I_n + J_{j,n})_{n\geq \max(m,j+2)}$. The claim follows by combining \Cref{fact: alexDualComplements} and \Cref{thm:gens from intro}.

Now assume that $\Delta(I_n)$ is not pure eventually. As above, the simplicial complex $\Delta(I_n + J_{j,n})$ consists of all faces of $\Delta(I_n)$ of dimension at most $j$. However, since $\Delta(I_n)$ need not be pure, there may be facets in $\Delta(I_n)$ and hence of $\Delta(I_n + J_{j,n})$ of dimension strictly less than $j$. By Lemma \ref{lem: allFacetsAppear}, we can choose $N$ large enough so that for all $i< j$, the number of facets of $\Delta(I_n)$ of dimension $i$ is constant whenever $n \geq N$.

Let $i = \max\left\{\dim (F) ~\mid~ F \text{ a facet of }\Delta(I_n) \text{ of dimension} < j\right\}$. 
Let $$K_n = I_n + J_{j,n} + \langle \xx^F ~\mid~ F \in \Delta(I_n)\text{ facet},\;\dim F=i\rangle.$$
It is easy to see that $(K_n)_{n\geq N}$ is a $\Sym$-invariant chain of ideals.

We may repeat the process of replacing $I_n + J_{j,n}$ with $K_n$, resulting in a removal of facets of dimension $i'$ for some fixed $i'<i$. This process will terminate and $\Delta(K_n)$ will have no facets of dimension $<j$ and in particular will be pure. Now the claim follows from the pure case.
\end{proof}


\section*{Acknowledgments}

This project began during the online workshop ``Research Encounters in Algebraic and Combinatorial Topics'' (REACT) in February 2021. We are grateful to  Alessio D'Al\'i, Mariel Supina and Lorenzo Venturello for organizing this workshop and bringing our team together.
We also thank Heide Gluesing-Luerssen for showing us Example \ref{ex: one-orbit}.
The first author was partially supported by the NSF GRFP under Grant No. DGE-1650441. The fourth author was partially supported by Simons Foundation grant \#636513. The fifth author was partially supported by NSF grant DMS-2053288.



\end{document}